\documentclass{amsart}
\usepackage{ae,lmodern} 
\usepackage{graphicx}
\usepackage{mathtools}
\usepackage{amsmath}
\usepackage{amsfonts}
\usepackage{amssymb}
\usepackage{color}
\usepackage{tabularx}
\usepackage[left=3.25cm,right=3.25cm,top=2.5cm,bottom=2.5cm]{geometry}
\usepackage{multicol}
\usepackage{url}
\usepackage{tikz}
\usepackage{enumitem}
\usepackage{setspace}
\usepackage{tikz-cd}
\usepackage{changepage}
\usetikzlibrary{matrix}
\usetikzlibrary{arrows}
\usepackage{pgf,tikz,pgfplots}
\pgfplotsset{compat=1.15}
\usepackage{mathrsfs}
\usetikzlibrary{arrows}
\usepackage[toc,page]{appendix}

\makeatletter
\@namedef{subjclassname@2020}{%
  \textup{2020} Mathematics Subject Classification}
\makeatother

\theoremstyle{plain}
\newtheorem{theorem}{Theorem}[section]
\newtheorem{corollary}[theorem]{Corollary}
\newtheorem{lemma}[theorem]{Lemma}
\newtheorem{proposition}[theorem]{Proposition}
\newtheorem*{theoremAH}{Theorem AH}
\newtheorem{theoremA}{Theorem}
\newtheorem{propositionA}[theoremA]{Proposition}

\theoremstyle{definition}
\newtheorem{definition}[theorem]{Definition}
\newtheorem*{definition*}{Definition}

\newtheorem{example}[theorem]{Example}

\newtheorem{remark}[theorem]{Remark}
\newtheorem{slogan}[theorem]{Slogan}

\newcommand{\Akp}{\mathbb{A}_{\Bbbk'}}

\newcommand{\PC}{\mathbb{P}_{\mathbb{C}}}
\newcommand{\PR}{\mathbb{P}_{\mathbb{R}}}
\newcommand{\Pk}{\mathbb{P}_{\Bbbk}}
\newcommand{\Pkp}{\mathbb{P}_{\Bbbk'}}
\newcommand{\Pkb}{\mathbb{P}_{\overline{\Bbbk}}}
\newcommand{\N}{\mathbb{N}}
\newcommand{\Z}{\mathbb{Z}}
\newcommand{\Q}{\mathbb{Q}}
\newcommand{\R}{\mathbb{R}}
\newcommand{\K}{\mathbb{K}}

\renewcommand{\L}{\mathbb{L}}

\newcommand{\C}{\mathbb{C}}
\newcommand{\s}{\mathbb{S}} 
\newcommand{\D}{\mathcal{D}}
\newcommand{\T}{\mathbb{T}}
\newcommand{\GC}{\mathbb{G}_{m,\mathbb{C}}} 
\newcommand{\Gk}{\mathbb{G}_{m,\Bbbk}}

\newcommand{\Gkp}{\mathbb{G}_{m,\Bbbk'}} 
\newcommand{\Gkth}{\mathbb{G}_{m,\Bbbk_3}} 
\newcommand{\Gko}{\mathbb{G}_{m,\Bbbk_1}} 
\newcommand{\Gkt}{\mathbb{G}_{m,\Bbbk_2}}

\newcommand{\Gkb}{\mathbb{G}_{m,\overline{\Bbbk}}}

\newcommand{\GR}{\mathbb{G}_{m,\mathbb{R}}} 
\newcommand{\Spec}{\mathrm{Spec}} 
\newcommand{\Proj}{\mathrm{Proj}} 
\newcommand{\Frac}{\mathrm{Frac}} 
\renewcommand{\H}{\mathrm{H}} 
\newcommand{\Aut}{\mathrm{Aut}} 
\newcommand{\Hom}{\mathrm{Hom}} 
\newcommand{\Div}{\mathrm{div}} 
\newcommand{\Gal}{\mathrm{Gal}} 
\newcommand{\Gl}{\mathrm{GL}} 
\newcommand{\RC}{{R}_{\mathbb{C} /  \mathbb{R}}} 
 
\newcommand{\Bbbkb}{\overline{\Bbbk}} 
 
\newcommand{\colim}{\mathrm{colim}}

\title[Torus actions on affine varieties over characteristic zero fields]{Torus actions on affine varieties \\ over characteristic zero fields}
 
\author{Pierre-Alexandre Gillard}

\thanks{The IMB receives support from  the EIPHI Graduate School (contract ANR-17-EURE-0002).
}

\address{Institut de Math\'{e}matiques de Bourgogne, UMR 5584 CNRS, Universit\'{e} Bourgogne Franche-Comt\'{e}, F-21000 Dijon, France}
\email{pierre-alexandre.gillard@u-bourgogne.fr}



\begin{document}

\begin{abstract}
Using Galois descent tools, we extend the Altmann-Hausen presentation of normal affine algebraic varieties endowed with an effective torus action over an algebraically closed field of characteristic zero to the case where the ground field is an arbitrary field of characteristic zero. In this context, the acting torus may have non-trivial torsors and we need additional data to encode such varieties. Finally, we focus on affine varieties endowed with a two-dimensional torus action and we provide a method for determining when a torsor is trivial, in which case the Altmann-Hausen presentation simplifies.
\end{abstract}

\maketitle

\tableofcontents

\section{Introduction}

\subsection{Aims and scope} Normal affine algebraic varieties endowed with a torus action over an algebraically closed field of characteristic zero admit a geometrico-combinatorial presentation due to Altmann and Hausen in \cite{Alt}. In this context, an $n$-dimensional torus is an algebraic group $\T$ isomorphic to $\Gk^n$.  Over a non closed field $\Bbbk$, a $\Bbbk$-torus is defined as an algebraic group $T$ such that, for some $n \in \N^*$, 
$$T_{\Bbbkb}:= T \times_{\Spec(\Bbbk)} \Spec(\Bbbkb) \cong \Gkb^n.$$ 
Over $\Bbbkb$ all tori are \emph{split}, i.e. isomorphic to $\Gkb^n$ for some $n \in \N^*$, but over $\Bbbk$ a torus may be non-split and may have non-trivial torsors. The existence of non split tori and non-trivial torsors imply to insert additional data in the geometrico-combinatorial presentation over non closed fields. 

Over the past few years, using Galois descent methods, the Altmann-Hausen presentation was extended in different directions. In 2015, Langlois  in \cite{Langlois, LangloisCorrected}  extended it  for complexity one torus actions over an arbitrary field in any characteristic using a specific construction. In the $\C/\R$ setting,  an $\R$-torus is a product of copies of the split torus $\GR$, of the real circle $\s^1= \Spec(\R[x,y]/(x^2+y^2-1))$, and of the Weil restriction $\RC(\GC)$. Contrary to $\GR$ and $\RC(\GC)$, the circle $\s^1$ has a non-trivial torsor, namely  $\Spec(\R[x,y]/(x^2+y^2+1))$.  In 2018, Dubouloz, Liendo and Petitjean focused on normal affine $\R$-varieties endowed with $\s^1 $-actions  in   \cite{Lien, Petit}, and they observed that an additional datum is needed to encode these varieties because of the existence of the non-trivial $\s^1$-torsor.  More generally, a geometrico-combinatorial presentation of normal affine $\R$-varieties endowed with an arbitrary $\R$-torus action based on \cite{Alt} was established in 2021 by the author in \cite{PA}.  These three presentations agree in the sense that they use a similar language and that we recover the first  two presentations  from the last one in the $\C/\R$ setting. 

As noted in \cite{PA}, it is natural and reasonable to expect that a general presentation of normal affine varieties endowed with torus actions over arbitrary fields of characteristic zero can be obtained by combining Altmann-Hausen theory for split torus actions together with appropriate Galois descent methods: this is precisely the goal of this article and this is the reason why we work over characteristic zero fields. We extend the Altmann-Hausen presentation to arbitrary fields of characteristic zero using a similar approach to the one adopted in \cite{PA}. However, in this context, the Galois extension $\Bbbkb/\Bbbk$ may be infinite and the description of $\Bbbk$-tori is more involved (except in dimension 1). We give a complete description of normal affine varieties endowed with a torus action over arbitrary fields of characteristic zero, and we recover the one given in \cite{PA} by letting $\Bbbk=\R$.

\subsection{Overview of the article} \label{IntroOverview}

\subsubsection*{Altmann-Hausen presentation} For the convenience of the reader, we  briefly recall the Altmann-Hausen presentation over an algebraically closed field $\Bbbk$ of characteristic zero. Let $\T \cong \Gk^n$ be a  $\Bbbk$-torus with character lattice $M$.   Let $\omega$ be a   full dimensional cone in $M_{\Q}:= M \otimes_{\Z} \Q$, let $N$ be the dual lattice of $M$, let $Y$ be a normal semi-projective $\Bbbk$-variety (i.e. $\H^0(Y, \mathcal{O}_Y)$ is a finitely generated $\Bbbk$-algebra and the affinization morphism $Y \to \Spec(\H^0(Y, \mathcal{O}_Y)) $ is projective), and let $\D := \sum \Delta_i \otimes D_i$ be a \emph{proper $\omega^{\vee}$-polyhedral divisor}. This means that the $D_i$ are prime divisors on $Y$ and the coefficients $\Delta_i$ are convex polyhedra in $N_{\Q}$ having $\omega^{\vee}$ as tail cone. Then, for every $m \in \omega \cap M$, we can evaluate $\D$ in $m$ to obtain a Weil $\Q$-divisor $\D(m) := \sum \text{min}\{ \langle m | \Delta_i \rangle \} \otimes D_i$. The proper condition on $\D$ means that, for all $m \in \omega \cap M$, $\D(m)$ is a semi-ample Cartier $\Q$-divisor, and for all $m \in \mathrm{Relint}(\omega) \cap M$, $\D(m)$ is big. From the triple $(\omega, Y, \D)$, Altmann and Hausen construct an $M$-graded $\Bbbk$-algebra:
\begin{equation*}
A[Y, \D]:= \bigoplus_{m \in \omega \cap M} \H^0(Y, \mathcal{O}_Y(\D(m))) \subset \Bbbk(Y)[M].
\end{equation*}  
The main results of \cite{Alt} can be summarized as follows:

\begin{theoremAH}[{\cite[Theorems 3.1 and 3.4]{Alt}}] Let $\T$ be a $\Bbbk$-torus.
\begin{enumerate}[leftmargin=0.75cm, label=(\roman*)]
\item The  scheme $X[Y, \D] := \Spec(A[Y, \D])$  is a normal $\Bbbk$-variety  endowed with an effective  $\mathbb{T}$-action.
\item Let $X$ be an affine normal $\Bbbk$-variety endowed with an effective $\mathbb{T}$-action. There exists a datum  $(\omega^{\vee}, Y, \D)$  as above such that $X \cong X[Y, \D]$ as $\T$-varieties.
\end{enumerate}
\end{theoremAH}

As we will see in Proposition \ref{splittorus}, this presentation extends \emph{mutatis mutandis} for split $\Bbbk$-torus actions over any field $\Bbbk$ of characteristic zero. 
In contrast, if the acting $\Bbbk$-torus $T$ is non-split, we need to insert additional data in the geometrico-combinatorial presentation. Namely, the linear Galois action on the character lattice of $T_{\Bbbkb}$,  and a datum that encodes a certain $T$-torsor.    

\subsubsection*{Galois descent tools} To extend the Altmann-Hausen presentation over an arbitrary field of characteristic zero, we use the language of $\Bbbk$-structures on $\Bbbkb$-varieties. Let $\Gamma := \Gal(\Bbbkb/\Bbbk)$ be endowed with the Krull topology. Roughly, a \emph{$\Bbbk$-structure} on an algebraic $\Bbbkb$-variety $X$ is a continuous action  $\Gamma \times X \to X$  such that the following diagram commutes for all $\gamma \in \Gamma$ (see Definition \ref{defstructure}): \hfill
\begin{center}
\begin{tikzpicture}
  \matrix (m) [matrix of math nodes,row sep=0.85em,column sep=6em,minimum width=2em]
  {          X     &     X                      \\
    \Spec(\Bbbkb)  &    \Spec(\Bbbkb)    \\
 };
  \path[-stealth]
    (m-1-1) edge node [above] {\small $\sigma_{\gamma}$} (m-1-2)
           edge  (m-2-1)
    (m-2-1) edge node [above] {\small $\Spec(\gamma)$} (m-2-2)
		(m-1-2) edge (m-2-2);
\end{tikzpicture}
\end{center}
A  \emph{$\Bbbk$-group structure} $\tau$ on an   algebraic $\Bbbkb$-group $G$ is a $\Bbbk$-structure on $G$ that preserves the group structure of $G$ (see Definition \ref{defGroupStructure}).  Let us note  that a  $\Bbbk$-group structure $\tau$ on a $\Bbbkb$-torus $\mathbb{T}$ corresponds to a linear $\Gamma$-action   $\tilde{\tau}$ on its character lattice $M$. 
There is an equivalence of categories between the category of quasi-projective algebraic $\Bbbk$-varieties (resp. algebraic $\Bbbk$-groups) and the category of quasi-projective algebraic $\Bbbkb$-varieties endowed with a $\Bbbk$-structure (resp.  algebraic $\Bbbkb$-groups endowed with a $\Bbbk$-group structure); see Proposition \ref{eqcat} for a precise statement. Therefore we will often  write $(X, \sigma)$ to refer to a  $\Bbbk$-variety and  $(\T, \tau)$ to refer to a $\Bbbk$-torus.  

\subsubsection*{Generalized Altmann-Hausen presentation}  Let $(\T , \tau)$ be a $\Bbbk$-torus and   let $M$ be the character lattice of $\T$. One of the main result of this article is the following one; it is the analog of Theorem AH over an arbitrary field $\Bbbk$ of characteristic zero. This result is based on a Galois-equivariant version of a method mentioned in \cite[\S 11]{Alt} (see Proposition \ref{ToricDown}).

\begin{theoremA}[{see Theorem \ref{Result}}] \label{ResultIntro} Let $(\T , \tau)$ be a $\Bbbk$-torus. 
\begin{enumerate}[leftmargin=0.75cm, label=(\roman*)]
\item Let $(\omega^{\vee}, Y, \D)$ be an AH-datum over $\Bbbkb$. If there exists a $\Bbbk$-structure $\sigma_Y$ on $Y$ and a map $h : \Gamma \to \Hom( \omega \cap M, \Bbbkb(Y)^*)$ such that
\begin{center}
\begin{tikzpicture}
  \matrix (m) [matrix of math nodes,row sep=0.5em,column sep=0.15em,minimum width=1em]
  { \forall m \in \omega \cap M, \ \forall \gamma \in \Gamma, &  {\sigma_Y}_{\gamma}^*(\D(m))=\D(\tilde{\tau}_{\gamma}(m)) + \Div_Y(h_{\gamma}(\tilde{\tau}_{\gamma}(m))); \text{ \ \  and \ \ }   \\
  \forall m \in \omega \cap M, \  \forall \gamma_1, \gamma_2  \in \Gamma,  &  h_{\gamma_1}(m) {\sigma_Y}_{\gamma_1}^{\sharp}(h_{\gamma_2}(\tilde{\tau}^{-1}_{\gamma_1}(m))) = h_{\gamma_1 \gamma_2}(m), \text{ \ \  \ \ \ \ \ }  \\
 };
\end{tikzpicture}
\end{center}
then $X[Y, \D]$ admits a $\Bbbk$-structure $\sigma_{X[Y, \D]}$ such that $(\T, \tau)$ acts on $( X[Y, \D], \sigma_{X[Y, \D]})$.
\smallskip
\item Let $(X, \sigma)$ be a normal affine $\Bbbk$-variety endowed with a $(\T, \tau)$-action. There exists an AH-datum $(\omega^{\vee}, Y, \D)$ together with a $\Bbbk$-structure $\sigma_Y$ on $Y$ and a map $h : \Gamma \to \Hom( \omega \cap M, \Bbbkb(Y)^*)$ satisfying the above conditions, such that $(X, \sigma) \cong ( X[Y, \D], \sigma_{X[Y, \D]})$ as $(\T, \tau)$-varieties.
\end{enumerate}
\end{theoremA}

The geometrico-combinatorial presentation over an arbitrary field of characteristic zero contains a cocycle $h$. This cocycle encodes a torsor (see Remark \ref{Ltorus} for a precise statement). Therefore, this presentation simplifies, that is we \emph{can take $h=1$}, if there are no non-trivial torsors; it is the case for instance for quasi-trivial torus actions (see Definition \ref{DefQuasiTrivialNorm} and Proposition \ref{PropQuasiTrivialAH}). In the $\C/\R$ setting, this simplification occurs when the acting torus does not contain $\s^1$-factors.  Furthermore, if the acting torus is split, then we recover \emph{mutatis mutandis} the Altmann-Hausen presentation of \cite{Alt} (see Proposition \ref{splittorus}).

From the Altmann-Hausen presentation of split $\Bbbk$-torus actions, we get as expected an effective method to compute an AH-datum of a $T$-action on a $\Bbbk$-variety $X$ (see Corollary \ref{ResultSplit}). We  consider  a finite Galois extension $\Bbbk'/\Bbbk$ that splits $T$. Then, we determine an AH-datum for the $T_{\Bbbk'}$-action on $X_{\Bbbk'}$ using a $\Gal(\Bbbk'/\Bbbk)$-equivariant embedding $X_{\Bbbk'}$ in some $\Akp^n$ (see Proposition \ref{ToricDown}), and then we  deduce an AH-datum for the $T$ action on a $X$.

\subsubsection*{Two-dimensional torus actions}  

Tori of dimension $n$ over an arbitrary field $\Bbbk$ are classified by the conjugacy classes of finite subgroup of $\Gl_n(\Z)$. Therefore, a one-dimensional $\Bbbk$-torus, up to isomorphism,  is either a split $\Bbbk$-torus or a norm one $\Bbbk$-torus (see Definition \ref{DefQuasiTrivialNorm}). For instance, if $\Bbbk= \R$, the norm one $\R$-torus is $\s^1$. The classification of the conjugacy classes of finite subgroup of $\Gl_2(\Z)$ is well known and there is an explicit description of two-dimensional $\Bbbk$-tori given by Voskresenski in \cite{VosI} (see Proposition \ref{Qtori}). Based on this classification, there is a complete description of the Galois cohomology set that classify torsors appearing in Theorem \ref{ResultIntro}  in \cite[Theorems 5.3 \& 5.5]{Elizondo} (see Theorem \ref{ElizondoH1}). 

In Section \ref{Section2Torus}, using birational geometry tools for surfaces, we relate  the torsor encoded by $h$ to a certain Del Pezzo surface. Let  $(\T, \tau)$ be a two-dimensional $\Bbbk$-torus, and let $(\T, \sigma)$ be a  $(\T, \tau)$-torsor (see Remark \ref{RmkTorsors}). Then, applying  a  $(\T, \tau)$-equivariant MMP to a $(\T, \tau)$-equivariant smooth compactification of $(\T, \sigma)$, we obtain the following proposition.

\begin{propositionA}[{Proposition \ref{birgeom}}] 
Let  $(\T, \tau)$ be a two-dimensional $\Bbbk$-torus and let $(\T, \sigma)$ be a  $(\T, \tau)$-torsor. There is a minimal toric Del Pezzo $\Bbbk$-surface $X$ birational to $(\T, \sigma)$  that contains $(\T, \sigma)$ as  a $(\T, \tau)$-stable dense open subset. Moreover, we have the following possibilities for $X$:
\begin{enumerate}[leftmargin=0.75cm, label=(\roman*)]
\item $X$ is a $\Bbbk$-form  of $\Pkb^2$; or
\item $X$ is a $\Bbbk$-form of the Del Pezzo $\Bbbkb$-surface of degree $6$; or
\item $X$ is a $\Bbbk$-form  of $\Pkb^1 \times \Pkb^1$.
\end{enumerate}
\end{propositionA}   
 
Finally, in the Appendices, we give some technical results used to prove the main results of this article.

\bigskip
 
\noindent \textbf{Acknowledgments.}
The author  is grateful to Adrien Dubouloz and Ronan Terpereau. He thanks them for their valuable help and for proofreading this article.

\bigskip

\section{Notation} \label{SectionNotation}
 
\subsection{Setting} \textbf{In this paper, all fields have characteristic zero}. Let $\Bbbk$  be an arbitrary field. Throughout  the  paper, we call  \emph{$\Bbbk$-variety} a separated   and geometrically integral  scheme of finite type over  $\Bbbk$, and an \emph{algebraic $\Bbbk$-group} is a finite type group scheme over $\Bbbk$.  
The group of regular automorphisms of a $\Bbbk$-variety $X$ is denoted by $\Aut(X)$, and the group of regular group automorphisms of an  algebraic $\Bbbk$-group   $G$ is denoted by $\Aut_{gr}(G)$.
If a $\Bbbk$-torus acts on an affine variety, \textbf{we always assume that the torus action is effective}.   

Let $\Bbbkb$ be a fixed  algebraic closure of $\Bbbk$; the field extension $\Bbbkb/\Bbbk$ is Galois. The Galois group $\Gal(\Bbbkb/\Bbbk)$ is a profinite group endowed with the \emph{Krull topology}, which is defined as follows (see \cite[Chapter IV, \S 17]{Morandi}). A subset $S$ of $\Gal(\Bbbkb/\Bbbk)$ is open if $S=\emptyset$ or if $S= \bigcup_i \gamma_i \Gal(\Bbbk_i/\Bbbk)$ for some $\gamma_i \in \Gal(\Bbbkb/\Bbbk)$ and for some finite Galois extensions $\Bbbk_i/\Bbbk$ in $\Bbbk'$. 

In this article, we consider a non-necessarily finite  Galois extension $\Bbbk'/\Bbbk$ in $\Bbbkb$, and we denote $\Gamma:= \Gal( \Bbbk'/\Bbbk)$ its Galois group. Most of the results are stated in this general setting, but sometimes we specify $\Bbbk'=\Bbbkb$ or $\Bbbk'/\Bbbk$ to be a finite Galois extension. Recall that the set of finite Galois extension $\Bbbk_i/\Bbbk$ in $\Bbbk'$ equipped with the inclusion is a directed partially ordered set, and that $\Gamma \cong \lim \Gal( \Bbbk_i / \Bbbk)$, where the \emph{limit} is taken over all finite Galois extension $\Bbbk_i/\Bbbk$ in $\Bbbk'$ (see \cite[\S 0BMI]{Stack}, or \cite[\S I.2.4]{Berhuy}).

\subsection{Convex geometry} From here on, $N$ denotes a lattice, i.e.~a finitely generated free abelian group,  and $M := \Hom_{\Z}(N, \Z)$ denotes its dual lattice. The associated $\Q$-vector spaces are denoted by $N_{\Q} := N \otimes_{\Z} \Q$ and $M_{\Q} := M \otimes_{\Z} \Q$ respectively, and the the corresponding pairing by 
\begin{equation*}
 M \times N \to \Z, \ \ \ (u,v) \mapsto \langle u,v \rangle := u(v).
\end{equation*}
Let us recall some results of \cite[\S 1.2]{Fult}. 
A subset $\omega_N \subset N_{\Q} $ is called a \emph{convex polyhedral cone} if there exists  a finite set  $S \subset N_{\Q}$ such that 
\begin{equation*}
\omega_N  = \text{Cone}(S) := \left\{ \left. \sum \lambda_v v \ \right| \ {v \in S}, \ \lambda_v \in {\Q}_{\geq 0} \right\} \subset N_{\Q}. 
\end{equation*}
A cone $\omega_N$ is   \emph{pointed}  if it contains no line. 
For us, a \textsl{cone} in $N_{\Q}$ is always a convex polyhedral cone.
 The dual cone of $\omega_N$ is defined by $\omega_N^{\vee} := \{ \left. u \in M_{\Q} \ \right| \  \forall v \in \omega_N, \ \langle u | v \rangle \geq 0 \}$; it is a cone in $M_{\Q}$. Let  $\omega_N$ be a cone in $N_{\Q}$.   A \emph{face} $\tau_N$ of $\omega_N$ is given by $ \tau_N = \omega_N \cap u^{\perp}$, for some $u \in \omega_N^{\vee}$, where
$u^{\perp} := \{ v \in \omega_N \ | \ \forall u \in \omega_N^{\vee},  \ \langle u,v \rangle  = 0 \}$. 
Recall that a face of a cone is a cone. The relative interior  $\text{Relint}(\omega_N)$ of a cone $\omega_N  $ is obtained  by removing all   proper faces from $\omega_N$. 

A \emph{quasifan} $\Lambda$ in $ N_{\Q}$ (or in $M_{\Q}$) is a finite collection of cones in $ N_{\Q}$ (or in $M_{\Q}$)  such that,  for any $\lambda \in \Lambda$, all the faces   of $ \lambda$ belong to $\Lambda$,  and for any   $\lambda_1, \lambda_2 \in \Lambda$, the intersection $\lambda_1 \cap \lambda_2$ is a face of both $\lambda_i$. 
The \emph{support} of a quasifan is the union of all its cones. A quasifan is called a \emph{fan} if all its cones are pointed.

\section{Galois descent and algebraic tori} \label{SectionGDAT}

In this section, $\Bbbk$ is an arbitrary field, $\Bbbk'/\Bbbk$  is a non-necessarily finite Galois extension in $\Bbbkb$, and   $\Gamma:= \Gal(  \Bbbk'/\Bbbk)$ is the corresponding Galois group.  We  recall basic definitions and well-known facts about $\Bbbk$-structures on $\Bbbk'$-varieties and $\Bbbk$-group structures on algebraic $\Bbbk'$-groups   in view of studying torus actions on $\Bbbk$-varieties.

\subsection{Galois descent} \label{SubsectionGD} 

Let us briefly recall the classical correspondence between quasi-projective $\Bbbk$-varieties and quasi-projective $\Bbbk'$-varieties endowed with a $\Bbbk$-structure. For finite Galois extensions, this correspondence is a direct consequence of étale descent, and for infinite Galois extension, it is a consequence of fpqc descent (see \cite{SGA1}). 

Let $\Bbbk'/\Bbbk$ be a finite Galois extension. A descent datum on a quasi-projective $\Bbbk'$-variety $X$ corresponds to a $\Bbbk$-structure on $X$, that is an algebraic semilinear $\Gamma$-action $\sigma : \Gamma    \to \Aut( X_{/ \Bbbk})$. For a translation of the language of descent used in \cite{SGA1} to the language of $\Bbbk$-structure, see  \cite[\S 6.2 Example B]{Neron}. This translation is based on the following $\Bbbk'$-algebra isomorphism
$$\Bbbk' \otimes_{\Bbbk} \Bbbk' \to \prod_{\Gamma} \Bbbk'; \ a \otimes b \mapsto \left( a \gamma(b) \right)_{\gamma \in \Gamma}.$$   
If $\Bbbk'/\Bbbk$ is an infinite Galois extension, this $\Bbbk'$-algebra morphism is an injective morphism that is not surjective. Therefore, in this context, a  descent datum on a quasi-projective $\Bbbk'$-variety $X$ corresponds to a $\Bbbk$-structure on $X$ as defined in Definition  \ref{defstructure} (see \cite{RVanDB}, \cite[2.4 Cohomologie galoisienne]{SGA41/2}, and \cite[Lemme 2.12]{BorelSerre}). 

Every $\Bbbk'$-variety $X$ can be viewed as a $\Bbbk$-scheme   via the composition of its structure morphism $X \to \Spec(\Bbbk')$ with the morphism $\Spec(\Bbbk') \to \Spec(\Bbbk)$ induced by the inclusion $\Bbbk \hookrightarrow \Bbbk'$. We denote this scheme by $X_{/ \Bbbk}$. The Galois group $\Gamma$ acts on $\Spec(\Bbbk')$, where the action is induced by the field automorphisms $(\gamma : \Bbbk' \to \Bbbk') \in \Gamma$.   

\begin{definition}  \label{defstructure}    
\begin{enumerate}[leftmargin=0.75cm, label=(\roman*)]
\item A  \emph{$\Bbbk$-form} of a $\Bbbk'$-variety $X$ is a $\Bbbk$-variety $X_0$ together with an isomorphism $(X_0)_{\Bbbk'} := X_0 \times_{\Spec(\Bbbk)} \Spec(\Bbbk') \cong X$ of $\Bbbk'$-varieties. By abuse of notation  we will often write: $X_0$ is a $\Bbbk$-form of $X$ instead of $ ( X_0,\cong)$. 
\item  A \emph{$\Bbbk$-structure}  $\sigma$ on a $\Bbbk'$-variety $X$  is a  \emph{(continuous) algebraic semilinear} $\Gamma$-action $ \sigma : \Gamma    \to \Aut( X_{/ \Bbbk}),  \gamma   \mapsto \sigma_{\gamma}$ satisfying $\sigma_{\gamma_1 \gamma_2} = \sigma_{\gamma_2} \circ \sigma_{\gamma_1}$ for all $\gamma_1, \gamma_2 \in \Gamma$, which means that:
\begin{itemize}
\item    for each $\gamma \in \Gamma$, the following diagram commutes: 
\begin{center}
\begin{tikzpicture}
  \matrix (m) [matrix of math nodes,row sep=1em,column sep=4em,minimum width=2em]
	{ X       &                   &     X                      \\
 \Spec(\Bbbk') & &  \Spec(\Bbbk') \\
};
  \path[-stealth]
    (m-1-1) edge node [above] {\small$\sigma_{\gamma}$} (m-1-3)
           edge  (m-2-1)
    (m-2-1) edge node [above] {\small$\Spec( \gamma )$} (m-2-3)
		(m-1-3) edge  (m-2-3);
\end{tikzpicture}
\end{center}
\item there exists a finite Galois extension $\Bbbk_1/\Bbbk$ in $\Bbbk'$, and a $\Bbbk_1$-form $X_1$ of $X$ such that the restriction of $\sigma$ to $\Gal(\Bbbk'/\Bbbk_1)$ coincides with the natural $\Gal(\Bbbk'/\Bbbk_1)$-action on $(X_1)_{\Bbbk'}  \cong X$.   
\end{itemize}
\item Two $\Bbbk$-structures $\sigma$ and $\sigma'$ on $X$ are \emph{equivalent} if there exists   $\varphi \in \Aut(X) $ such that $\sigma'_{\gamma} = \varphi \circ \sigma_{\gamma}  \circ \varphi^{-1}$ for all $\gamma \in \Gamma$.
\item A   \emph{$\Bbbk$-morphism} between two $\Bbbk'$-varieties $X$ and $X'$ with $\Bbbk$-structures $ \sigma $ and $ \sigma' $ is a morphism of $\Bbbk'$-varieties $f : X \to X'$ such that $\sigma'_{\gamma} \circ f = f \circ \sigma_{\gamma}$, as morphisms of $\Bbbk$-schemes, for all $\gamma \in \Gamma$.
\end{enumerate}
\end{definition}

Observe that, if $\Bbbk'/\Bbbk$ is a finite Galois extension, the definition of a $\Bbbk$-structure of Definition \ref{defstructure} coincides with the usual one for finite Galois extension. 
We refer to the Appendix \ref{Appendix1} for more details on the link between infinite and finite Galois theory. 

If $X$ is a $\Bbbk$-variety, then $X_{\Bbbk'}$ is endowed with a canonical $\Bbbk$-structure  given by  $\gamma \in \Gamma \mapsto id \times \Spec(\gamma)$. Furthermore, if $X'$ is another $\Bbbk$-variety and $f : X \to X'$ is a morphism of $\Bbbk$-varieties, then $f \times id : X_{\Bbbk'} \to  X'_{\Bbbk'}$ is a morphism of $\Bbbk'$-varieties. The next proposition is a reformulation of the effectiveness of fpqc descent in the category of quasi-projective varieties \cite[Corollaire 7.9]{SGA1}.

\begin{proposition}  \label{eqcat}
The functor $X \mapsto  (X_{\Bbbk'}, \sigma) $, where $\sigma_{\gamma}:= id \times \Spec(\gamma)$ for all $\gamma \in \Gamma$,  induces an equivalence of categories between the  category of pairs $(X, \sigma)$ consisting of a   quasi-projective $\Bbbk'$-variety $X$ endowed with a $\Bbbk$-structure $\sigma$, and the category of quasi-projective $\Bbbk$-varieties. Moreover, equivalent $\Bbbk$-structures on $X$ correspond to isomorphic $\Bbbk$-forms of $X$.  
\end{proposition}

Using this equivalence, we denote  $(X, \sigma) \mapsto X/\Gamma$  the corresponding inverse functor, and by abuse of notation we write $(X, \sigma)$ to refer to an algebraic $\Bbbk$-variety. 

Once we know the existence of a $\Bbbk$-structure $\sigma$ on a $\Bbbk'$-variety $X$, a Galois cohomology set can be used to parametrize the equivalence classes of $\Bbbk$-structure on $X$. First, note that $\Aut(X)$ is a discrete topological abstract group equipped with a \emph{continuous $\Gamma$-action} 
$$\gamma \cdot \varphi := \sigma_{\gamma}^{-1} \circ \varphi \circ \sigma_{\gamma}, \  \forall \gamma \in \Gamma \text{ and } \forall  \varphi \in \Aut(X).$$
Here \emph{continuous} means: for all $\varphi \in \Aut(X)$, the set $\text{Stab}_{\Gamma}(\varphi):=\{ \gamma \in \Gamma  \ | \ \gamma \cdot \varphi = \varphi \}$ is an open subgroup of $\Gamma$. Then, assume there exists another $\Bbbk$-structure $\sigma'$ on $X$. Observe that the map $c : \Gamma \to \Aut(X)$ defined for all $\gamma \in \Gamma$ by $c_{\gamma} := \sigma'_{\gamma} \circ \sigma_{\gamma}^{-1}$ is cocycle; that is a continuous map such that for all $\gamma_1, \gamma_2 \in \Gamma$,  $c_{\gamma_1 \gamma_2} = c_{\gamma_1} \circ (\gamma_1 \cdot c_{\gamma_2})$.   Two cocycles $c$ and $c'$ are \emph{equivalent} if there exists $\varphi \in \Aut(X)$ such that for all $\gamma \in \Gamma$, $c'_{\gamma} = \varphi^{-1} \circ  c_{\gamma} \circ ( \gamma \cdot \varphi)$.  The set of cocycles modulo this equivalence relation is the \emph{first pointed set of Galois cohomology}   $\H_{cont}^1(\Gamma,  \Aut(X))$, and
we have an isomorphism of pointed set 
$$\H_{cont}^1 \left( \Gamma, \Aut(X) \right) \cong  \colim \ \H^1 \left( \Gal(\Bbbk_i/\Bbbk), \Aut(X/{\Gal(\Bbbk'/\Bbbk_i)}) \right),$$ 
where the \emph{colimit} is taken over all finite Galois extensions $\Bbbk_i/\Bbbk$ in $\Bbbk'$ (see \cite[\S I.2.2, Proposition 8]{Serre}, see also \cite[Theorems I.2.8, II.3.33 and Lemma II.3.3]{Berhuy}).

\begin{proposition}[{\cite[\S III.1.3, Proposition 5]{Serre}, also \cite[Theorem 14.91]{Gortz}}] \label{PropH1}
Let $X$ be a $\Bbbk'$-variety equipped with a $\Bbbk$-structure $\sigma$. There is a bijection
\begin{align*}
\H_{cont}^1 \left( \Gamma, \Aut(X) \right) \simeq \{ \text{equivalence classes of $\Bbbk$-structure on $X$} \}; \ \
\left( \gamma \mapsto c_{\gamma} \right)  \mapsto \left( \gamma \mapsto c_{\gamma} \circ \sigma_{\gamma} \right)
\end{align*}
that sends the trivial cocycle  $\gamma \mapsto id$  to the equivalence class of $\sigma$.
\end{proposition}

We have similar definitions and properties for algebraic groups, those are always quasi-projective varieties from Chevalley's theorem (see \cite[Corollary 1.2]{ConradChevalleyThm}). These similar definitions and properties are consequences of fpqc descent in the category of algebraic groups (see \cite[Corollaires 7.7 et 7.9]{SGA1}).

\begin{definition} \label{defGroupStructure}
\begin{enumerate}[leftmargin=0.75cm, label=(\roman*)]
\item Let $G$ be an algebraic $\Bbbk'$-group. An algebraic $\Bbbk$-group $G_0$ together with an isomorphism $(G_0)_{\Bbbk'} := G_0 \times_{\Spec(\Bbbk)} \Spec(\Bbbk') \cong G$ is called a   \emph{$\Bbbk$-form} of $G$. 
\item  A  \emph{$\Bbbk$-group} \emph{structure} $\tau$ on an algebraic $\Bbbk'$-group $G$  is a $\Bbbk$-structure $\tau : G \to G$ such that the multiplication $G \times G \to G$, the inverse $G \to G$ and the unity $\Spec(\Bbbk') \to G$ are $\Bbbk$-morphisms. 
\item Two $\Bbbk$-group structures $\tau$ and $\tau'$ on $G$ are \emph{equivalent} if there exists   $\varphi \in \Aut_{gr}(G) $ such that $\tau'_{\gamma} = \varphi \circ \tau_{\gamma} \circ \varphi^{-1}$ for all $\gamma \in \Gamma$.  
\item A   \emph{$\Bbbk$-morphism} between two   algebraic $\Bbbk'$-groups $G$ and $G'$  with $\Bbbk$-structures $ \tau $ and $  \tau' $ is a morphism of  algebraic $\Bbbk'$-groups $f : G \to G'$ such that $\tau'_{\gamma} \circ f = f \circ \tau_{\gamma}$ as morphisms of $\Bbbk$-schemes for all $\gamma \in \Gamma$.
\end{enumerate} 
\end{definition}
 
There is an equivalence of categories between the category of pairs $(G, \tau)$ consisting of an algebraic $\Bbbk'$-group endowed with a  $\Bbbk$-group structure, and  the category of algebraic $\Bbbk$-groups. This equivalence induces a one-to-one correspondence between the $\Bbbk$-forms of G, up to isomorphism in the category of   algebraic $\Bbbk$-groups, and the equivalence classes of $\Bbbk$-group structures on G.  Furthermore, if $\tau$ is a given $\Bbbk$-group structure on $G$, then $\Bbbk$-forms of $G$ are classified by $\H_{cont}^1( \Gamma, \Aut_{gr}(G))$ (see \cite[\S III.1.3, Corollary of Proposition 5]{Serre}).

\subsection{Tori} \label{SubsectionRT} 

See for instance \cite[\S 30.3.3]{ManinCubic} for a survey of the theory of tori and of  torsors. We give here the essential results for this article. See the Appendix \ref{Appendix2} for more details on $\Bbbk$-tori.

\begin{definition}      
 A \emph{$\Bbbk$-torus} $T$ is an algebraic $\Bbbk$-group  such that $T_{\Bbbkb} \cong \Gkb^n$ for some integer $n$. It is called a \emph{split $\Bbbk$-torus} if $T \cong \Gk^n$ for some integer $n$. 
\end{definition}

Since the $\Bbbk$-algebra $\Bbbk[T]$ of any $\Bbbk$-torus $T$ is finitely generated, there exists a finite  Galois   extension $\Bbbk_1/\Bbbk$ in $\Bbbkb$ such that $T_{\Bbbk_1}$ is a split torus (see \cite[Proposition 1.2.1]{Ono}, or \cite[Lemmas 0EXM and 09DT]{Stack}).

Observe that there is a one-to-one correspondence between $n$-dimensional split $\Bbbk$-tori and rank $n$ lattices. Indeed, to a lattice $M \cong \Z^n$ is associated the split torus $\Spec(\Bbbk[M])$, where $\Bbbk[M] := \{ \sum_{\text{finite}} a_m \chi^m \ | \ a_m \in \Bbbk,  m \in M \}$ is a $\Bbbk$-algebra such that $\chi^{m+m'} = \chi^m \chi^{m'}$. Conversely, to a split torus $\T$ is associated its character lattice $\Hom_{gr}( \T, \Gk) \cong \Z^n$. In the rest of the paper, we will write $\T = \Spec(\Bbbk[M])$ to refer to a split $\Bbbk$-torus. 
 
\begin{remark}[$\Gamma$-action on the character lattice] \label{representation}
Let $\tau$ be a $\Bbbk$-group structure on $\T=\Spec(\Bbbk'[M])$. Then $\tau$  induces a $\Gamma$-representation $\tilde{\tau} : \Gamma  \to \Gl(M)$, and a dual $\Gamma$-representation   $\hat{\tau} : \Gamma \to \Gl(N)$,   such that $\tilde{\tau}_{\gamma_1 \gamma_2} = \tilde{\tau}_{\gamma_1} \circ \tilde{\tau}_{\gamma_2}$ and  $\hat{\tau}_{\gamma_1 \gamma_2} = \hat{\tau}_{\gamma_2} \circ \hat{\tau}_{\gamma_1}$ for all $\gamma_1, \gamma_2 \in \Gamma$. Indeed, one can shows that
\begin{align*}
\tau^{\sharp} :\Gamma   \to \Aut_{\Bbbk} \left( \Bbbk'[M] \right), \ \ \   
\gamma   \mapsto  \left(   a_m \chi^m  \mapsto  \gamma(a_m) \chi^{\tilde{\tau}_{\gamma}(m)} \right).
\end{align*} 
Fix a $\Z$-basis of $M$. We get isomorphisms $M \cong \Z^n$, $\Gl(M) \cong \Gl_n(\Z)$,  and $\T \cong \Gk^n$, for some $n \in \N^*$. Two $\Gamma$-representations $\rho$ and $\rho'$ on $\Gl_n(\Z)$ are equivalent if there exists $P \in \Gl_n(\Z)$ such that $\rho'(\gamma) = P \circ \rho(\gamma) \circ P^{-1}$ for all $\gamma \in \Gamma$.
There is a one-to-one correspondence between  $\Bbbk$-group structures on $\Gkp^n$ and faithful $\Gamma$-representations  in $\Gl_n(\Z)$. Furthermore, equivalent classes of $\Bbbk$-group structures on $\Gkp^n$ correspond to equivalent classes of $\Gamma$-representations in $\Gl_n(\Z)$.
Finally, note that $n$-dimensional $\Bbbk$-tori are classified by $\H_{cont}^1( \Gal(\Bbbkb/\Bbbk), \Aut_{gr}(\Gkb^n))$.
\end{remark}

\begin{example} \label{ExSplit} The  split $\Bbbk$-torus $\Gk^n = \Spec(\Bbbk[M])$ corresponds to $(\Gkp^n, \tau_0)$, where $\tau_0$ is the $\Bbbk$-group structure on $\Gkp^n = \Spec(\Bbbk'[M])$ such that the induced $\Gamma$-action on $M$ is trivial:
\begin{align*}
\tau_0^{\sharp} : \Gamma  \to \Aut_{\Bbbk} \left( \Bbbk'[M] \right); \ \
\gamma  \mapsto  \left( a_m \chi^m  \mapsto \gamma(a_m) \chi^m \right).
\end{align*}
\end{example}

\begin{definition}[{\cite[III.3]{ColliotThNotes}, \cite[\S 2]{ColliotTh}}] \label{DefQuasiTrivialNorm}
\begin{enumerate}[leftmargin=0.75cm, label=(\roman*)]
\item A $\Bbbk$-torus $T$ is a \emph{Weil restriction $\Bbbk$-torus} if there exists a finite Galois extension $\Bbbk'/\Bbbk$ in $\Bbbkb$ such that $T \cong  \text{R}_{\Bbbk'/\Bbbk}(\Gkp^n)$ (see for instance \cite[\S 7.6]{Neron} and \cite[\S A.5]{ConradBrian}).
\item Let $T$ be $\Bbbk$-torus that splits over a finite Galois extension $\Bbbk_1/\Bbbk$. It is  called a \emph{quasi-trivial $\Bbbk$-torus} if there exits a basis of the character lattice $M$ of $T_{\Bbbk_1} \cong \Gko^n$ that is permuted by $\Gal(\Bbbk_1/\Bbbk)$. A quasi-trivial $\Bbbk$-torus is isomorphic to a product of tori  $\text{R}_{\Bbbk'/\Bbbk}(\Gkp^n)$, where $\Bbbk'$ are finite Galois extensions of $\Bbbk$ in $\Bbbkb$. Note that a split $\Bbbk$-torus is a quasi-trivial $\Bbbk$-torus. 
\item Let $\Bbbk'/\Bbbk$ be a finite Galois extension of degree $d$. There is a surjective $\Bbbk$-group morphism $N_{\Bbbk'/\Bbbk} : \text{R}_{\Bbbk'/\Bbbk}(\Gkp) \to \Gk$. The kernel is a torus of dimension $d-1$ that splits over $\Bbbk'$,  called the \emph{norm  one torus}  and denoted by $\text{R}^{(1)}_{\Bbbk'/\Bbbk}(\Gkp)$. That is,  the following sequence is exact:
$$1 \longrightarrow \text{R}^{(1)}_{\Bbbk'/\Bbbk} \left( \Gkp \right) \longrightarrow \text{R}_{\Bbbk'/\Bbbk} \left( \Gkp \right) \longrightarrow  \Gk \longrightarrow 1.$$
A $\Bbbk$-torus $T$ is a \emph{norm one $\Bbbk$-torus} if there exists a finite Galois extension $\Bbbk'/\Bbbk$ in $\Bbbkb$ such that $T \cong  \text{R}^{(1)}_{\Bbbk'/\Bbbk}(\Gkp^n)$. 
\end{enumerate} 
\end{definition}

These tori are also the building blocks of any $\R$-torus and of any two-dimensional $\Bbbk$-torus (see Proposition \ref{Qtori}). 

\begin{example} \label{ExEquationNormOne} Let $\Bbbk'/\Bbbk$ be a finite Galois extension of degree $d$.  Then $N_{\Bbbk'/\Bbbk}(\Bbbk) : (\Bbbk')^* \to \Bbbk^*$ corresponds to the usual norm map of a field extension (see for instance \cite[Chapter VI \S 5]{Lang}). As usual, if  $\omega_1, \dots, \omega_d$ is $\Bbbk$-basis of $\Bbbk'$, we denote $\Xi := \omega_1x_1 + \dots + \omega_n x_d \in \Bbbk'[x_1, \dots , x_d]$ and $N_{\Bbbk'/\Bbbk}(\Xi) \in \Bbbk[x_1, \dots, x_d]$ be the induced form of degree $d$ (see \cite{Flanders}).   Then, 
$$\text{R}^{(1)}_{\Bbbk'/\Bbbk} \left( \Gkp \right) = \Spec \left( \Bbbk[x_1, \dots, x_d] / (N_{\Bbbk'/\Bbbk}(\Xi) -1) \right).$$
For instance, the real circle $\s^1:=\Spec(\R[x,y]/(x^2+y^2-1))$ is the norm one $\R$-torus.
\end{example}


\begin{example} \label{ExNorm1deg3} 
Let $\Bbbk:= \C(t)$ and $\Bbbk':= \Bbbk[u]/(u^3-t)$. The extension $\Bbbk'/\Bbbk$ is Galois with Galois group $\Gamma \cong \Z/3\Z$. We obtain:
$$\text{R}_{\Bbbk'/\Bbbk} \left( \Gkp \right) = \Spec \left( \dfrac{\Bbbk[x_1,y_1,z_1,x_2,y_2,z_2]}{(x_1x_2 + ty_1z_2 + tz_1y_2 -1, x_1y_2 + x_2y_1 + t z_1z_2, x_1z_2+x_1z_2+x_2z_1)}\right).$$
The corresponding norm one torus of dimension 2 is
$$\text{R}^{(1)}_{\Bbbk'/\Bbbk} \left( \Gkp \right) = \Spec(\Bbbk[x,y,z]/(x^3 + ty^3 + t^2z^3 - 3txyz-1)).$$
\end{example}

\begin{example}[{Biquadratic field extension}] \label{ExampleBiquadratic} 
Recall that $\Bbbk'/\Bbbk$ is a Galois extension with Galois group isomorphic to the Klein group $\Z/2\Z \times \Z/2\Z$ if and only if $\Bbbk'=\Bbbk(\sqrt(a), \sqrt(b))$ for some $a,b \in \Bbbk$ such that none of $a$, $b$, or $ab$ is a square in $\Bbbk$. So, let $a$, $b \in \Bbbk$ satisfying this property. Then, $\{ 1, \sqrt{a}, \sqrt{b}, \sqrt{ab}\}$ is a $\Bbbk$-basis of $\Bbbk'$ and the Galois group is $\{ id, \gamma_1, \gamma_2, \gamma_1 \gamma_2\}$, where 
$\gamma_1 : \sqrt{a} \mapsto -\sqrt{a},  \ \sqrt{b}  \mapsto  \sqrt{b}$ and
$\gamma_2 : \sqrt{a}  \mapsto \sqrt{a}, \ \sqrt{b}  \mapsto  -\sqrt{b}$.
Then, $N_{\Bbbk'/\Bbbk}(\Xi) = \prod_{\gamma \in \Gamma} (x +  \gamma  (\sqrt{a})y + \gamma  (\sqrt{b})z + \gamma  (\sqrt{ab})w)$. Therefore, the norm one $\Bbbk$-torus $\text{R}^{(1)}_{\Bbbk'/\Bbbk}(\Gkp)$ of dimension 3 is defined by the equation:
$$x^4+a^2y^4+b^2z^4+a^2b^2w^4-2ax^2y^2-2bx^2z^2-2ab(x^2w^2+y^2z^2)-2a^2by^2w^2-2ab^2z^2w^2 +8abxyzw=1.$$
\end{example}

\subsection{Torsors and toric varieties}

Split toric $\Bbbk$-varieties, that is $\Bbbk$-varieties endowed with an action of a split $\Bbbk$-torus $\T=\Spec(\Bbbk[M])$ having a dense open orbit isomorphic to $\T$ itself,   are determined by a fan in $N_{\Q}$; the $\Q$-vector space associated to the cocharacter lattice of $\T$ (see for instance \cite[\S 2.1]{Elizondo}). If the acting  torus is a non-necessarily split $\Bbbk$-torus, we cannot define such lattices, and the definition of toric $\Bbbk$-varieties we use in this paper involves the notion of torsor.

\begin{definition} \label{DefTorsor}
Let $T$ be a $\Bbbk$-torus. A $T$-torsor is a $\Bbbk$-variety $V$ together with a left action $\mu : T \times V \to V$ such that the map $(\mu, pr_2 ) : T \times V \to V \times V$ is an isomorphism. Two $T$-torsors $V$ and $V'$ are isomorphic if there exists a $T$-equivariant isomorphism $V \to V'$ of $\Bbbk$-varieties. A $T$-torsor is trivial if it is isomorphic to $T$ acting on itself by translation.
\end{definition}

A $T$-torsor $V$ is trivial if and only if the set $V(\Bbbk)$ of $\Bbbk$-points is not empty. In particular, any torsor over an algebraically closed field is trivial. Furthermore, by Hilbert's 90 Theorem,  any $\Gk^n$-torsor is trivial, and by \cite[III.3]{ColliotThNotes}, a quasi-trivial torus has no non-trivial torsors (see \cite[\S 2, \S 5]{ColliotTh}). 

\begin{remark} \label{RmkTorsors}
In our setting, there is an equivalent definition of a $T$-torsor. A $T$-torsor is a $\Bbbk$-variety $V$ such that $T_{\Bbbkb} \cong V_{\Bbbkb}$ and such that the action $\mu_{\Bbbkb} : T_{\Bbbkb} \times V_{\Bbbkb} \to V_{\Bbbkb}$ corresponds to the action by translation on $T_{\Bbbkb}$.
Since any $\Gkp^n$-torsor is trivial, a $(\Gkp^n, \tau)$-torsor  is a $\Bbbk$-variety $(\Gkp^n, \sigma)$ endowed with a $(\Gkp^n, \tau)$-action, where  $\sigma$ is a $\Bbbk$-structure on $\Gkp^n$.
\end{remark}

Since a $(\Gkp^n, \tau)$-torsor is a $\Bbbk$-form of $\Gkp^n$ viewed as a $\Gkp^n$-variety for the usual action by translation, from Proposition \ref{PropH1} and Remark \ref{RmkTorsors} we obtain the following result. 

\begin{proposition}  
Let   $(\T=\Spec(\Bbbk'[M]), \tau)$  be a $\Bbbk$-torus. Consider the continuous $\Gamma$-action on the group $\Aut^{\T}(\T)$ of $\T$-equivariant $\Bbbk'$-automorphisms of $\T$ given by $\gamma \cdot \varphi = \tau_{\gamma}^{-1} \circ \varphi \circ \tau_{\gamma}$. There is a functorial bijection 
\begin{align*}
\H_{cont}^1 \left( \Gamma, \Aut^{\T}(\T) \right) &\simeq \{ \text{isomorphism classes of $T$-torsors} \}; \\
\left( \gamma \mapsto \varphi_{\gamma} \right)  &\mapsto \big( \T, \ \left( \gamma \mapsto \varphi_{\gamma} \circ \sigma_{\gamma} \right) \big)
\end{align*}
that sends the trivial cocycle to a trivial $(\T=\Spec(\Bbbk'[M]), \tau)$-torsor.
\end{proposition}

Torsors of norm one tori are described by the next lemma.   

\begin{lemma}[{\cite[III.3]{ColliotThNotes}, \cite[\S 2]{ColliotTh}}] \label{H1Norm1}
Let $\Bbbk'/\Bbbk$ be a finite Galois extension, and let $T \cong \text{R}_{\Bbbk'/\Bbbk}^{(1)}(\Gkp)$. Then
$$\H^1  \left( \Gamma,  \Aut^{T_{\Bbbk'}}(T_{\Bbbk'}) \right) \cong \Bbbk^* / \mathrm{Im} \left( N_{\Bbbk'/\Bbbk}(\Bbbk) \right),$$
where $N_{\Bbbk'/\Bbbk}(\Bbbk) : \text{R}_{\Bbbk'/\Bbbk}(\Gkp)(\Bbbk) = (\Bbbk')^* \to \Gk(\Bbbk) = \Bbbk^*$ is the norm map obtained from the functor $N_{\Bbbk'/\Bbbk}$ mentioned in Definition \ref{DefQuasiTrivialNorm}. Furthermore (see \cite[Corollary 4.4.10]{GilleSZ}), if the extension $\Bbbk'/\Bbbk$ is cyclic, there is a canonical isomorphism 
$$\Bbbk^* / \mathrm{Im} \left( N_{\Bbbk'/\Bbbk}(\Bbbk) \right) \cong \mathrm{Br}(\Bbbk'/\Bbbk),$$
where $\mathrm{Br}(\Bbbk' /\Bbbk)$ is the kernel of the morphism $\mathrm{Br}(\Bbbk) \to \mathrm{Br}(\Bbbk')$ and  $\mathrm{Br}(\Bbbk)$ is the Brauer group of $\Bbbk$. 
\end{lemma}

\begin{example} \label{ExampleTorsorNorm}
We pursue Example \ref{ExEquationNormOne}.  
Let $\alpha \in \Bbbk^*$. Then, the $\Bbbk$-variety defined by the equation $N_{\Bbbk'/\Bbbk}(\Xi) = \alpha$ is a $\text{R}^{(1)}_{\Bbbk'/\Bbbk}(\Gkp)$-torsor, and any $\text{R}^{(1)}_{\Bbbk'/\Bbbk}(\Gkp)$-torsor has this form. By Lemma  \ref{H1Norm1}, this torsor is trivial if and only if $\alpha \in \mathrm{Im}( N_{\Bbbk'/\Bbbk}(\Bbbk))$. For instance:
\begin{enumerate}[leftmargin=0.75cm, label=(\roman*)]
\item If $\Bbbk'=\C$ and $\Bbbk=\R$, then $-1 \notin \mathrm{Im}( N_{\Bbbk'/\Bbbk}(\Bbbk))$ and $\Spec(\R[x,y] / (x^2 + y^2 + 1))$ is, up to isomorphism, the only  non-trivial $\s^1$-torsor (see \cite[Proposition 3.1]{Lien});
\item If $\Bbbk'=\Bbbk(\sqrt{13}, \sqrt{17})$ and $\Bbbk=\Q$, then for instance $-1; 25; 49  \notin \mathrm{Im}( N_{\Bbbk'/\Bbbk}(\Bbbk)$ (see Example \ref{ExampleBiquadratic} and \cite[Example 5.3]{Cassels});
\end{enumerate}
\end{example}

\begin{example} \label{ExNorm1deg2'}
We pursue Example \ref{ExNorm1deg3}. Since $\Bbbk$ is an extension of $\C$ of transcendence degree one, then $\mathrm{Br}(\Bbbk)= \{ 1 \}$ (see \cite[Chapter X, \S 7]{SerreLocal}). Therefore $\text{R}^{(1)}_{\Bbbk'/\Bbbk}(\Gkp)$ has no non-trivial torsors.
\end{example}

\begin{definition}
A \emph{toric $\Bbbk$-variety}  is a normal $\Bbbk$-variety $V$ such that there exits a $\Bbbk$-torus $T$ acting effectively on $V$ with a dense open orbit $U$.  Note that $U$ is a $T$-torsor.\\
A toric $\Bbbk$-variety $V$ is called a \emph{split toric $\Bbbk$-variety} if $T$ is a split $\Bbbk$-torus (i.e. $T \cong \T = \Spec(\Bbbk[M])$). In this case, since there is no non-trivial torsors (by Hilbert's 90 Theorem), $\T$ acts on $V$ with a dense open orbit $U \cong \T$ and there exists a fan $\Lambda$ in $N_{\Q}$ such that $V \cong V_{\Lambda}$ (see \cite[\S 2.1]{Elizondo}).
\end{definition}

\section{Altmann-Hausen presentation over arbitrary fields} \label{SectionAH}

In this section, we extend the Altmann-Hausen presentation of torus actions on affine varieties over algebraically closed fields to arbitrary fields (see Theorem \ref{Result}). We show that the cocycle appearing in the presentation over non closed fields encodes a torsor, and a simpler presentation is possible when this torsor is trivial. Theorem \ref{Result} generalizes \cite[Theorems 4.3 \& 4.6]{PA}, those treat the $\C/\R$ case, and \cite[Theorem 5.10]{Langlois} that focus on complexity one torus actions (see Remark \ref{RmkLanglois}). 

\subsection{Altmann-Hausen presentation over an arbitrary field}
The Altmann-Hausen presentation is introduced in \cite{Alt}. See also \cite[\S 3 \& \S 4]{PA} for a detailed summary of the Altmann-Hausen presentation.

Let $\Bbbk$ be an arbitrary field, let $\Bbbk'=\Bbbkb$, and let $\Gamma:=\Gal(\Bbbkb/\Bbbk)$.   
Let $\T =\Spec(\Bbbkb[M])$ be a $\Bbbkb$-torus,  let $N$ be  the dual lattice of $M$, and let $X$ be an affine $\Bbbkb$-variety endowed with a $\T$-action $\mu$. The $\T$-action on $X$ corresponds to an $M$-grading of its coordinate ring,
$$\Bbbkb[X] = \bigoplus_{m \in M} \Bbbkb[X]_m,$$ 
the spaces $\Bbbkb[X]_m$ consisting of semi-invariant regular functions on $X$, that is $\Bbbkb[X]_m := \{ f \in \Bbbkb[X] \ | \ \mu^{\sharp}(f) = \chi^m\otimes f \}$,  where $\chi^m : \T \to \Gkb$ is the character associated to $m \in M$. The weight cone of the $\T$-action is the cone $\omega_M$ of $M_{\Q}:= M \otimes_{\Z} \Q$ spanned by the set $\{ m \in M \ | \ \Bbbkb[X]_m \neq 0 \}$.

Let $\omega_N$ be a pointed cone in $N_{\Q}= N \otimes_{\Z} \Q$, let $Y$ be a normal semi-projective $\Bbbkb$-variety, and let $\D := \sum \Delta_i \otimes D_i$ be a proper $\omega_N$-polyhedral divisor (see Section \ref{IntroOverview}). From the triple $(\omega_N, Y, \D)$ (called an AH-datum), Altmann and Hausen construct an $M$-graded $\Bbbkb$-algebra:
$$A[Y, \D] := \bigoplus_{m \in \omega_N^{\vee}  \cap M} \H^0(Y, \mathcal{O}_Y(\D(m))) \subset \Bbbkb(Y) [M].$$
By \cite[Theorems 3.1]{Alt}, the affine scheme $X[Y,\D] := \Spec(A[Y, \D])$ is a normal $\Bbbkb$-variety endowed with a $\T$-action of weight cone $\omega_N^{\vee}$ (see Theorem AH).

We  will state one of the main results of this article. We extend the Altmann-Hausen presentation over an arbitrary field. In this setting, the acting torus is non-necessarily split.  In the next definition, we adapt the notion of AH-datum to our setting.

\begin{definition}[{Generalized AH-datum}] 
Let $\T = \Spec(\Bbbkb[M])$ be a $\Bbbkb$-torus and   let $\tau$ be a $\Bbbk$-group structure on $\T$.  A \emph{generalized AH-datum} $(\omega_N, Y, \D, \sigma_Y, h)$ over $\overline{\Bbbk}$ is an AH-datum  $(\omega_N, Y, \D)$ over $\overline{\Bbbk}$ together with a $\Bbbk$-structure $\sigma_Y$ on $Y$ and with a map $h : \Gamma \to \Hom( \omega_N^{\vee} \cap M, \Bbbkb(Y)^*)$ such that
\begin{equation} \let\veqno\eqno \label{eqdiv}
\forall m \in \omega_M^{\vee} \cap M, \ \forall \gamma \in \Gamma, \ \ \ \ \ \ \    {\sigma_Y}_{\gamma}^*(\D(m))=\D(\tilde{\tau}_{\gamma}(m)) + \Div_Y\left(h_{\gamma} \left(\tilde{\tau}_{\gamma}(m) \right) \right), \text{ \ \  and \ \ } 
\end{equation}
\begin{equation} \let\veqno\eqno \label{eqdiv2}
\forall m \in \omega_M^{\vee} \cap M, \  \forall \gamma_1, \gamma_2  \in \Gamma, \ \ \ \ \ \ \ \      h_{\gamma_1}(m) {\sigma_Y}_{\gamma_1}^{\sharp}\left(h_{\gamma_2}\left(\tilde{\tau}^{-1}_{\gamma_1}(m)\right)\right) = h_{\gamma_1 \gamma_2}(m). \ \ \ \ \  \ \ \ \ \  \ \ \ \ \ 
\end{equation}
\end{definition}

\begin{remark} \label{Ltorus} Let $(\T = \Spec(\Bbbkb[M]), \tau)$ be a $\Bbbk$-torus, let $(\omega_N, Y, \D, \sigma_Y, h)$ be a generalized AH-datum, let $\L:= \Bbbkb(Y)$, and let $\K=\L^{\Gamma}$. By Lemma \ref{Artin}, the extension $\L/\K$ is Galois with Galois group $\Gamma$. Let $G:= \Hom_{gr}( M, \L^*)$ be endowed with the continuous $\Gamma$-action  $\gamma \cdot f := {\sigma_Y}_{\gamma}^{\sharp} \circ f \circ \tilde{\tau}_{\gamma}^{-1}$. Note that the map $h$ mentioned in  Theorem \ref{Result} is a cocycle, that is $h \in \H^1_{cont}( \Gamma, G)$. 
Let $\T_{\L}:=\Spec(\L[M])$ be the $\L$-torus associated to $M$, and let $\tau_{\L}$ be the $\K$-group structure on $\T_{\L}$ induced by $\tilde{\tau}$. The cocycle $h$ encodes a $(\T_{\L}, \tau_{\L})$-torsor. Indeed, one can easily  shows that 
$$\H^1_{cont}(\Gamma,  G) \cong \H^1_{cont} \left( \Gamma,  \Aut^{\T_{\L}} \left( \T_{\L} \right) \right).$$
\end{remark}

From a generalized AH-datum over $\overline{\Bbbk}$, we can easily construct an affine $\Bbbk$-variety endowed with a $(\T, \tau)$-action (see the proof of the next theorem in Section \ref{SectionProof}).

\begin{theorem}[Torus actions on normal affine varieties over arbitrary fields] \label{Result}  
Let $\Bbbk$ be a field, and let   $\Gamma := \Gal(\Bbbkb/\Bbbk)$. Let $\T = \Spec(\Bbbkb[M])$ be a $\Bbbkb$-torus, and let $\tau$ be a $\Bbbk$-group structure on $\T$. 
\begin{enumerate}[leftmargin=0.75cm, label=(\roman*)]
\item Let $(\omega_N, Y, \D, \sigma_Y, h)$ be a generalized AH-datum over $\overline{\Bbbk}$.  The affine $\T$-variety $X[Y, \D]$ admits a $\Bbbk$-structure $\sigma_{X[Y, \D]}$ such that $(\T, \tau)$ acts on $( X[Y, \D], \sigma_{X[Y, \D]})$.
\item Let $(X, \sigma)$ be a normal affine $\Bbbk$-variety endowed with a $(\T, \tau)$-action. There exists a generalized AH-datum $(\omega_N, Y, \D, \sigma_Y, h)$ such that $(X, \sigma) \cong ( X[Y, \D], \sigma_{X[Y, \D]})$ as $(\T, \tau)$-varieties.
\end{enumerate}
\end{theorem}

\begin{example} 
Let $(X, \sigma)$ be an affine $(\T, \tau)$-toric variety, where $\T=\Spec (\Bbbkb[M])$. Let $\omega_M$ be the weight cone of the $\T$-action on $X$. In this case, $Y=\Spec(\Bbbkb)$ is endowed with the $\Bbbk$-structure $\gamma \mapsto \Spec(\gamma)$, $\D$ is trivial, and $h : \Gamma \to \Hom(\omega_M \cap M, \Bbbkb^*)$. Therefore, the Altman-Hausen presentation is given by $\omega_M$ (that encodes the toric $\T$-variety $X$ since $\Bbbkb[X] \cong A[Y,\D] = \Bbbkb[\omega_M \cap M]$), and by $h$ (that encodes the $\Bbbk$-structure $\sigma$ on $X$ compatible with $\tau$). Let $\gamma \in \Gamma$, we will see in the proof of Theorem \ref{Result} that   $\sigma_{\gamma}^{\sharp} : \Bbbkb[\omega_M \cap M]   \to \Bbbkb[\omega_M \cap M];  a_m \chi^m \mapsto \gamma(a_m) h_{\gamma}(\tilde{\tau}_{\gamma}(m)) \chi^{\tilde{\tau}_{\gamma}(m)}$.  Moreover, note that $\sigma_{\gamma}(X_{\T}) = X_{\T}$ for all $\gamma \in \Gamma$, where $X_{\T} \cong \T$ is the dense open orbit  of the toric $\Bbbkb$-variety $X$. Hence, $\sigma$ induces a $\Bbbk$-structure $\sigma_{\T}$ on $X_{\T}$ and $(X_{\T}, \sigma_{\T})$ is a $(\T, \tau)$-torsor. 
\end{example}

\subsection{Proof of Theorem \ref{Result}} \label{SectionProof}

The next result, which generalizes \cite[Proposition 4.1]{PA}, is a key ingredient to construct an AH-datum from  a $T$-action on a $\Bbbk$-variety. Here, $\Bbbk'/\Bbbk$ is a non-necessarily finite Galois extension in $\Bbbkb$ of Galois group $\Gamma$. The next proposition will be used with $\Bbbk'=\Bbbkb$ in the proof of Theorem \ref{Result}, but with a finite Galois extension $\Bbbk'/\Bbbk$ in Corollary \ref{ResultSplit}.

\begin{proposition}[Downgrading torus action] \label{ToricDown}
Let $X$ be an affine $\Bbbk'$-variety endowed with an action of the $d$-dimensional torus $\T=\Spec(\Bbbk'[M])$.       Let $\sigma$ be a $\Bbbk$-structure on $X$ and let $\tau$ be a $\Bbbk$-group structure on $\T$. If the $\Bbbk$-torus $(\T, \tau)$ acts on $(X, \sigma)$, then there is  $n \in \N, n \geq d$ such that:
\begin{enumerate}[leftmargin=0.75cm, label=(\roman*)]
\item There is a  $\Bbbk$-group  structure $ \tau' $ on $\Gkp^n $   that extends to a  $\Bbbk$-structure $ \sigma' $ on $\Akp^n $; 
\item $(\T, \tau)$ is a closed subgroup of $(\Gkp^n, \tau')$; and
\item $(X, \sigma)$ is a closed subvariety of   $(\Akp^n, \sigma')$,   and   $(X, \sigma) \hookrightarrow (\Akp^n, \sigma')$ is  $(\T, \tau)$-equivariant. Moreover, $X$ intersects  the dense open orbit of   $\Akp^n$  for the natural $\Gkp^n$-action, and the weight cone of $\Akp^n$ is the weight cone of $X$.
\end{enumerate}
\end{proposition}

\begin{proof} 
Let    $\Bbbk_1/\Bbbk$ be a finite Galois extension in $\Bbbk'$ that splits   the $\Bbbk$-torus $(\T, \tau)$. We have a tower of Galois extension $\Bbbk \subset \Bbbk_1 \subset \Bbbk'$. 
Let $H := \Gal(\Bbbk'/\Bbbk_1)$. By the Galois correspondence \cite[Theorem 0BML]{Stack},  $H$ is a normal subgroup of $\Gamma$ and $  \Gamma / H \cong \Gal(\Bbbk_1/\Bbbk)$.  
The $\Bbbk$-structure  $\sigma$  on $X$ restricts to a $\Bbbk_1$-structure  $\sigma_H := \sigma |_{H}$ on $X$, and  the $\Bbbk$-group structure  $\tau$  on $\T$ restricts to a $\Bbbk_1$-group structure  $\tau_H := \tau |_{H}$ on $\T$.  
Let  $X_1 := X / H$ and  $\T_1 := \T / H$ be the associated $\Bbbk_1$-varieties (see Proposition \ref{eqcat}).  We obtain an induced $\Bbbk$-structure $\sigma_1$ on $X_1$ and a $\Bbbk$-group structure $\tau_1$ on  $\T_1$ such that  $X_1/ \Gal(\Bbbk_1/\Bbbk) \cong X/\Gamma$ and  $\T_1/ \Gal(\Bbbk_1/\Bbbk) \cong \T/\Gamma$ (see Lemmas \ref{LinkFinite} and \ref{ForAll}). 
Moreover $(\T_1, \tau_1)$ acts on $(X_1, \sigma_1)$. Since $\Bbbk_1$ is a splitting field of $\T / \Gamma$,  $\T_1 \cong \Gko^n$, the $H$-action on $M$ is trivial and $\Bbbk'[M]^H = \Bbbk_1[M]$ (see Appendix \ref{Appendix2Split}). 

 \emph{(i)} The algebra $\Bbbk_1[X_1] = \Bbbk'[X]^H$ is finitely generated, so we can write  $\Bbbk_1[X_1] = \Bbbk_1[\tilde{g}_1,  \dots , \tilde{g}_k]$ with $\tilde{g}_i \in \Bbbk_1[X_1] \backslash \{ 0 \}$. Since  the split torus $\T_1 \cong \Spec(\Bbbk_1[M])$ acts on $X_1$, the $\Bbbk_1$-algebra $\Bbbk_1[X_1]$ is $M$-graded, that is $\Bbbk_1[X_1] = \bigoplus_{m \in M} \Bbbk_1[X_1]_m$. So,  there exists homogeneous elements $\tilde{g}_{i,j}$ such that $\tilde{g}_i = \tilde{g}_{i, 1} + \dots + \tilde{g}_{i, k_i}$. Note that $\Bbbk_1[X_1]$ is generated  as a $\Bbbk_1$-algebra  by $\{ \tilde{g}_{i,j}, {\sigma_1}_{\gamma}^{\sharp}(\tilde{g}_{i,j})  \ | \ \gamma \in \Gal(\Bbbk_1/\Bbbk) \}$.  Moreover, by Lemma \ref{IsomGradedStructure}, an homogeneous element is send to an homogeneous element by ${\sigma_1}_{\gamma}^{\sharp}  $ for all $\gamma \in \Gal(\Bbbk_1/\Bbbk)$. Hence, we can assume  there exists $n \in \N$ and homogeneous elements $g_i$ of degree $m_i \in M$, such that $\Bbbk_1[X_1] = \Bbbk_1[ {g}_1,  \dots ,  {g}_n]$ and such that  the set $\{ g_i \ | \ 1 \leq i \leq n\}$ is stable under the $\Gal(\Bbbk_1/\Bbbk)$-action   ${\sigma_1}^{\sharp}$ on $\Bbbk_1[X_1]$. 
 
We obtain a $\Gamma$-equivariant isomorphism $\Bbbk'[X]^H \otimes_{\Bbbk_1} \Bbbk' \cong \Bbbk'[X]$, where the $\Gamma$-action on the left hand side is given by $\gamma \mapsto {\sigma_1}_{\gamma}^{\sharp} \otimes \gamma$ for all $\gamma \in \Gamma$,  and where the $\Gamma$-action on $\Bbbk'[X]$ is given by ${\sigma}^{\sharp}$ (see Lemma \ref{LinkFinite}). Therefore, we can write  $\Bbbk'[X] = \Bbbk'[ {g}_1,  \dots ,  {g}_n]$,  where the set $\{ g_i \ | \ 1 \leq i \leq n\}$ is stable under the $\Gamma$-action $\sigma^{\sharp}$.

Let $\gamma \in \Gamma$, let ${\tau'}_{\gamma} $ and let $ {\sigma'}_{\gamma} $ be the maps induced by the  antilinear maps     
${\tau'}_{\gamma}^{\sharp} (x_i) = x_j$, 
	  and $ {\sigma'}_{\gamma}^{\sharp}  (x_i) = x_j$, where ${\sigma}_{\gamma}^{\sharp}(g_i) = g_j$.  This induces  a  $\Bbbk$-group structure $\tau'$ on $\Gkp^n = \Spec( \Bbbk'[x_1^{\pm1}, \dots, x_n^{\pm1}])$ and a   $\Bbbk$-structure $\sigma'$ on $\Akp^n = \Spec( \Bbbk'[x_1, \dots, x_n])$.  

 \emph{(ii)} The $\Bbbk'$-algebra  morphism
$ \psi :  \Bbbk'[x_1^{\pm 1},  \dots , x_n^{\pm 1}]    \to  \Bbbk'[M],  \ 
 x_i   \mapsto \chi^{m_i}$ 
is surjective since the $\T$-action    on $X$ is effective.  Since $(\T, \tau)$ acts on $(X, \sigma)$, 
$\psi$ is $\Gamma$-equivariant. So,  the   $\Bbbk$-algebra morphism   $\psi^{\Gamma} : \Bbbk'[x_1^{\pm 1},  \dots , x_n^{\pm 1}]^{\Gamma}    \to  \Bbbk'[M]^{\Gamma}$ is well defined and surjective.  Hence, $(\T, \tau)$ is a closed subgroup of $(\Gkp^n,  {\tau'})$. 
 
 \emph{(iii)} The $\Bbbk'$-algebra morphism   
$\varphi : \Bbbk'[x_1,  \dots , x_n]    \to \Bbbk'[X_1], \
 x_i   \mapsto g_i
$
is surjective and induces a $\Bbbk'$-algebra isomorphism  $\Bbbk'[g_1,  \dots , g_n] \cong  {\Bbbk'[x_1,  \dots , x_n]}/{\mathfrak{a}}$, with $ \mathfrak{a} = \text{Ker}(\varphi)$.   {Moreover}, the morphism $\varphi$ is $\Gamma$-equivariant. So,  the   $\Bbbk$-algebra morphism  $\varphi^{\Gamma} : \Bbbk'[x_1,  \dots , x_n]^{\Gamma}    \to \Bbbk'[X_1]^{\Gamma}$ is well defined and surjective. Hence,   $(X, \sigma)$ is a closed subvariety of   $(\Akp^n,  \sigma')$. 

Note that $\varphi$ is $\T$-equivariant, so the closed immersion    $X \hookrightarrow \Akp^n$ is $\T$-equivariant. 
Moreover, the comorphism of   the  $\T$-action   on $\Akp^n$ is given by:
\begin{equation*}
\tilde{\mu}^{\sharp}  :\Bbbk'[x_1, \dots, x_n]      \to \Bbbk'[M] \otimes  \Bbbk'[x_1, \dots, x_n],   \ \ \ 
 x_i  \mapsto \chi^{m_i} \otimes x_i
\end{equation*}
Then,  the  following diagram  commutes for all $\gamma \in \Gamma$:
\small
\begin{center}
\begin{tikzpicture}
\matrix (m) [matrix of math nodes,row sep=1.25em,column sep=4em,minimum width=2em]
 {
              & \Bbbk'[\Akp^n] \vphantom{\otimes \Bbbk'[\Akp^n]} &                   &  \Bbbk'[M] \otimes \Bbbk'[\Akp^n] \\
 \Bbbk'[\Akp^n] \vphantom{\otimes \Bbbk'[\Akp^n]}  &              & \Bbbk'[M] \otimes \Bbbk'[\Akp^n] &      \\
             &              &                   &        \\
				     &              &                   &        \\
	           &       \Bbbk'[X]  \vphantom{  \Bbbk'[M] \otimes}  &                   & \Bbbk'[M] \otimes \Bbbk'[X] \\
      \Bbbk'[X]  \vphantom{  \Bbbk'[M] \otimes}     &              & \Bbbk'[M] \otimes \Bbbk'[X] &     \\};
 \path[-stealth]
    (m-1-2) edge node  [above] {    {$\tilde{\mu}^{\sharp}$} } (m-1-4)
    (m-2-1) edge node  [above] { \ \ \ \ \ \ \ \ \ \ \ \ \ \ \  {$\tilde{\mu}^{\sharp} $ } } (m-2-3)
	  (m-5-2) edge [dashed] node  [below] {  {$\mu^{\sharp} \ \ \ \ \ \ \ \ \ \ \ \ \ \  $} } (m-5-4)
    (m-6-1) edge node  [below] {  {$\mu^{\sharp} $} } (m-6-3)
		(m-1-4) edge node  [left] { {${\tau}_{\gamma}^{\sharp} \times {\sigma'}_{\gamma}^{\sharp} \ \ $}  \ \ \ \  } (m-2-3)
    (m-5-4) edge node  [right] { \ \ \ \ \ \ {${\tau}_{\gamma}^{\sharp} \times  {\sigma}_{\gamma}^{\sharp} $}} (m-6-3)
    (m-5-2) edge [dashed] node  [above] { {${\sigma}_{\gamma}^{\sharp} \ \ \ \ \ $}}  (m-6-1)
		(m-1-2) edge node  [above] { {${\sigma'}_{\gamma}^{\sharp} \ \ \ \ \ \ \ $}}  (m-2-1)
    (m-2-1) edge node  [left] { {$\varphi$}} (m-6-1)
    (m-1-2) edge [dashed] node  [left] {    {$\varphi$} } (m-5-2)
	  (m-1-4) edge   node  [right] {  $id \times \varphi$ }(m-5-4)
		(m-2-3) edge node  [right] {   ${id \times \varphi}_{\vphantom{N_N} }$ }   (m-6-3);
\end{tikzpicture}
\end{center}
\normalsize
Hence, the morphism $\varphi$ is $(\T, \tau)$-equivariant, so $(X, \sigma)$ is a closed subvariety of $(\Akp^n,  {\sigma'})$, and $(X, \sigma) \hookrightarrow (\Akp^n,  {\sigma'})$ is $(\T, \tau)$-equivariant.  
Finally, note that for all $i \in \{1 , \dots, n \}$, $x_i \notin \mathfrak{a}$, hence $X$ intersects the dense open orbit of $\Gkp^n$. It follows that  the weight cone of $\Akp^n$ is the weight cone of $X$. 
\end{proof}

The next lemma is another key ingredient.

\begin{lemma} \label{Artin} Let $Y$ be a quasi-projective $\Bbbk'$-variety, let $\sigma$ be a $\Bbbk$-strucure on $Y$, and let $\L:= \Bbbk'(Y)$. Then,   $\sigma$ induces a faithful $\Gamma$-action on $\L$ by field automorphisms, and $\L / \L^{\Gamma}$ is a Galois extension of absolute Galois group $\Gal(\L / \L^{\Gamma}) \cong \Gamma$.   Moreover, $\L^{\Gamma} = \Bbbk(Y/\Gamma)$.
\end{lemma}

\begin{proof}   (Compare with \cite[Lemma 1.3.2]{Fried})
Since $Y$ is an integral scheme, for any $\Gamma$-stable affine open subset $U \subset Y$, the ring  $\mathcal{O}_Y(U)$ is   integral   and $\L := \Bbbk'(Y) =\Frac(\mathcal{O}_Y(U))$. Hence, $\sigma$ induces a  group homomorphism $\varphi : \Gamma  \to \Aut_{\Bbbk}(\L), \ \gamma  \mapsto \varphi_{\gamma} := {\sigma}_{\gamma}^{\sharp}$. 
 
Let $\Bbbk_1 / \Bbbk$   and let $\Bbbk_2/\Bbbk$ be   finite Galois extensions in $\Bbbk'$. Since there exists a finite Galois extension $\Bbbk_3/\Bbbk$  that contains $\Bbbk_1$ and $\Bbbk_2$   \cite[Lemmas 0EXM and 09DT]{Stack}, we can assume that $\Bbbk_2/\Bbbk$ is a finite Galois extension such that $\Bbbk \subset \Bbbk_1 \subset \Bbbk_2 \subset \Bbbk'$.  By the Galois correspondence \cite[Theorem  0BML]{Stack}, $H_2 :=\Gal(\Bbbk'/\Bbbk_2) $ is a normal subgroup of $H_1 := \Gal(\Bbbk'/\Bbbk_1)$.
 Consider the $\Bbbk_1$-variety   $Y_1:= Y/H_1$ and  the $\Bbbk_2$-variety  $Y_2:= Y/H_2$.  We have  $Y/\Gamma \cong Y_1 / \Gal(\Bbbk_1/\Bbbk) \cong Y_2 / \Gal(\Bbbk_2 / \Bbbk)$  (see Lemmas \ref{LinkFinite} and \ref{ForAll}).  
Moreover, since $\Gal(\Bbbk_2/\Bbbk_1)$ is a finite group, then $\Bbbk_1(Y_1)^{\Gal(\Bbbk_1/\Bbbk)} = \Bbbk_2(Y_2)^{\Gal(\Bbbk_2/\Bbbk)}$. Denote $\K:= \Bbbk_1(Y_1)^{\Gal(\Bbbk_1/\Bbbk)}$, this field does not depend on the finite Galois extension $\Bbbk_1/\Bbbk$. 
By \cite[Theorem 1.8 (Artin)]{Lang}, the field  extension  $\Bbbk_1(Y_1) / \K$ is a finite Galois extension of Galois group $\Gal(\Bbbk_1(Y_1)/\K) \cong \Gal(\Bbbk_1/\Bbbk)$. Furthermore, the field $\L:= \Bbbk'(Y)$ is the union of all above $\Bbbk_1(Y_1)$. Then (see \cite[Lemma 0BU2]{Stack}),  $\Gal(\Bbbk'/\Bbbk) = \lim  \Gal( \Bbbk_i / \Bbbk) \cong  \lim  \Gal(\Bbbk_i(Y_i) / \K) = \Gal(\L/\K)$, where the limits are over all finite Galois extension $\Bbbk_i/\Bbbk$ in $\Bbbk'$.  
\end{proof}

\smallskip

\begin{proof}[Proof of Theorem \ref{Result}] 
\emph{(i)} Let $(\omega_N, Y, \D, \sigma_Y, h)$ be a generalized AH-datum over $\overline{\Bbbk}$. By   \cite[Theorem 3.1]{Alt},  ${X[Y, \D]}:=\Spec(A[Y, \D])$ is a normal affine  $\Bbbkb$-variety   endowed with a $\T$-action, of   weight cone   $\omega_{N}^{\vee}$. This action is obtained from the following comorphism:
\begin{equation*}
\mu^{\sharp} : A[Y, \D]  \to \Bbbkb[M ] \otimes  A[Y, \D], \  
f \mathfrak{X}_m   \mapsto   \chi^{m} \otimes f \mathfrak{X}_m.
\end{equation*} 
We  now construct a $\Bbbk$-structure on ${X[Y, \D]}$ such that $(\T, \tau)$ acts on $({X[Y, \D]}, \sigma_{X[Y, \D]})$. Let $\gamma \in \Gamma$. By Condition (\ref{eqdiv}), we obtain isomorphisms of $A[Y, \D]_0$-modules:
\begin{align*}
{\alpha_{\gamma}}_m : \H^0( Y, \mathcal{O}_Y(\D({m}))) \mathfrak{X}_{ m}   & \to  \H^0( Y, \mathcal{O}_Y(\D(\tilde{\tau}_{\gamma}({m})))) \mathfrak{X}_{\tilde{\tau}_{\gamma}({m})} \\ 
f \mathfrak{X}_m   & \mapsto {\sigma_Y}_{\gamma}^{\sharp}(f) h_{\gamma}(\tilde{\tau}_{\gamma}({m}))\mathfrak{X}_{\tilde{\tau}_{\gamma}({m})}.
\end{align*}
These isomorphisms collect into an isomorphism  of $A[Y, \D]_0$-modules $\oplus_{m \in \omega_{N}^{\vee} \cap M}  \alpha_m$ on  $A[Y, \D]$. By Condition (\ref{eqdiv2}), the  latter isomorphism corresponds to a  $\Bbbk$-structure $\sigma_{X[Y, \D]}$ on ${X[Y, \D]}$. Finally, $(\T, \tau)$ acts on $({X[Y, \D]}, \sigma_{X[Y, \D]})$ since the following diagram commutes for all $\gamma \in \Gamma$:
\begin{center}
\begin{tikzpicture}
  \matrix (m) [matrix of math nodes,row sep=1.5em,column sep=7em,minimum width=2em]
  {
	  A[Y, \D]  &  \Bbbk[M] \otimes  A[Y, \D]        \\
	  A[Y, \D]  &  \Bbbk[M] \otimes  A[Y, \D]        \\};
 \path[-stealth]
   	(m-1-1) edge   node [above] {\small  $\mu^{\sharp}$}  (m-1-2)
		(m-2-1) edge   node [below] {\small  $\mu^{\sharp}$}  (m-2-2)
		(m-1-1) edge node [left] {\small  $  {{\sigma^{\sharp}}_{X[Y, \D]}}_{\gamma}$ }   (m-2-1)
		(m-1-2) edge node [right] {\small  $  {\tau }_{\gamma}^{\sharp} \otimes {\sigma^{\sharp}_{X[Y, \D]}}_{\gamma} $ }   (m-2-2);
\end{tikzpicture}
\end{center}

\emph{(ii)} The structure of the proof is the same as in \cite{PA}, but we give a sketch since the setting and the notation differ.
To determine a generalized AH-datum, we can proceed by a Galois-equivariant version of \emph{toric downgrading}  (introduced in \cite[\S 11]{Alt}) as in \cite{PA}. We embed $\T \times \Gamma$-equivariantly $X$ into some $\Akp^n$ such that $\T$ is a subtorus of $\Gkb^n$ (see Proposition \ref{ToricDown}). Denote by $\T_Y := \Gkb^n / \T$ the quotient torus. Let $M$ (resp. $M'$, $M_Y$) be the character lattice of $\T$ (resp. $\Gkb^n$, $\T_Y$). We obtain the exact sequences of lattices of Appendix \ref{Appendix2Morphism}.  Then, it is basically the same proof as in \cite{PA},  adding $\gamma$ in index everywhere and using \cite[Proposition 1.19]{Huruguen}, which can be extended to infinite Galois extension.
\end{proof}

\subsection{Galois cohomology, torsors and Altmann-Hausen presentation}

In the non-necessarily split version of the Altmann-Hausen construction (Theorem \ref{Result}),  a cocycle $h$ appears in the combinatorial presentation.  We will see that Theorem \ref{Result} simplifies (i.e we can take $h=1$) if this cocycle is equivalent to the trivial one.
 
Let $\Bbbk$ be a  field, let $\T=\Spec(\Bbbkb[M])$ be a  $\Bbbkb$-torus, and let $\tau$ be a $\Bbbk$-group structure on $\T$. Let $(X, \sigma)$ be a normal affine variety endowed with an action of $(\T, \tau)$, and let $(Y, \sigma_Y)$ be the $\Bbbk$-variety of Theorem \ref{Result}. Let $G:= \Hom_{gr}( M, \Bbbkb(Y)^*)$ be endowed with the continuous $\Gamma$-action  $\gamma \cdot f := {\sigma_Y}_{\gamma}^{\sharp} \circ f \circ \tilde{\tau}_{\gamma}^{-1}$ (see Remark \ref{Ltorus}). In this setting, we get the next result.  

\begin{lemma} \label{simpl}   
If $\H^1_{cont}( \Gamma, G) = \{ 1 \}$, then there exists  an $\omega_{N}$-pp divisor  $\D$  on $Y$  such that    
$$\forall  m \in \omega_{M} \cap M, \forall \gamma \in \Gamma, \ {\sigma_Y}_{\gamma}^*(\D(m)) = \D(\tilde{\tau}_{\gamma}(m)),$$
 and such that the varieties $(X, \sigma)$ and $(X[Y, \D], \sigma_{X[Y, \D]})$ are $(\T, \tau)$-equivariantly isomorphic.
\end{lemma}

\begin{proof} 
By Theorem \ref{Result}, there exists  a generalized AH-datum $(\omega_N, Y, \D, \sigma_Y, h)$ such that  the varieties $(X, \sigma_X)$ and $(X[Y, \D'], \sigma_{X[Y, \D']} )$ are $(\T, \tau)$-equivariantly isomorphic.
 Since $\H^1_{cont}(\Gamma, G ) = \{ 1 \}$,  the trivial cocycle  is equivalent to $h$, so there exists $g \in G$ such that $h_{\gamma}(m) = g^{-1}(m) {\sigma_Y}_{\gamma}^{\sharp}(g(\tilde{\tau}_{\gamma}(m)))$. Let   $m \in \omega_{M} \cap M$ and $\gamma \in \Gamma$, then:
$${\sigma_Y}_{\gamma}^*(\D(m)) = \D( \tilde{\tau}_{\gamma}(m) ) + \Div_Y(h_{\gamma}(\tilde{\tau}_{\gamma}(m)) ) = \D( \tilde{\tau}_{\gamma}(m) ) + {\sigma_Y}_{\gamma}^*\Div_Y(g(m))  -  \Div_Y(g(\tilde{\tau}_{\gamma}(m))) $$
$$\iff {\sigma_Y}_{\gamma}^* \left(\D(m) -  \Div_Y(g(m)) \right)= \D( \tilde{\tau}_{\gamma}(m) )   -  \Div_Y( g(\tilde{\tau}_{\gamma}(m))).$$
So, if $\D'$ is the  pp-divisor defined by $ \D'(m):= \D(m)-\text{div}_{Y}(g(m))$, then ${\sigma_Y}_{\gamma}^*(\D'(m))  = \D'( \tilde{\tau}_{\gamma}(m) )$, and the $M$-graded  algebras $A[Y, \D]$ and $A[Y, \D']$ are  isomorphic with respect to $\sigma_{X[Y,\D]}^{\sharp}$ and $\sigma_{X[Y,\D']}^{\sharp}$. Hence the varieties $(X, \sigma_X)$ and $(X[Y, \D'], \sigma_{X[Y, \D']})$ are $(\T, \tau)$-equivariantly isomorphic.
\end{proof}

Quasi-trivial $\Bbbk$-tori have no non-trivial torsors (see \cite[\S 2, \S 5]{ColliotTh}). Therefore, the Altmann-Hausen presentation simplifies for quasi-trivial $\Bbbk$-torus actions on normal affine varieties.

\begin{proposition}[{Quasi-trivial torus actions on normal affine varieties over arbitrary fields}]   \label{PropQuasiTrivialAH} 
Let $\Bbbk$ be a field, and denote   $\Gamma := \Gal(\Bbbkb/\Bbbk)$. Let $(\T = \Spec(\Bbbkb[M]),\tau)$ be a quasi-trivial $\Bbbk$-torus. 
\begin{enumerate}[leftmargin=0.75cm, label=(\roman*)]
\item Let $(\omega_N, Y, \D, \sigma_Y, h=1)$ be a generalized AH-datum over $\overline{\Bbbk}$.  Then the affine $\T$-variety $X[Y, \D]$ admits a $\Bbbk$-structure $\sigma_{X[Y, \D]}$ such that $(\T, \tau)$ acts on $( X[Y, \D], \sigma_{X[Y, \D]})$.
\item Let $(X, \sigma)$ be a normal affine $\Bbbk$-variety endowed with a $(\T, \tau)$-action. Then there exists a generalized AH-datum $(\omega_N, Y, \D, \sigma_Y, h=1)$ such that there is an isomorphism of $(\T, \tau)$-varieties $(X, \sigma) \cong ( X[Y, \D], \sigma_{X[Y, \D]})$.
\end{enumerate}
\end{proposition}

\begin{proof}
If $(\T, \tau)$ is a quasi-trivial $\Bbbk$-torus acting on a normal affine variety, then the $\K$-torus $(\T_{\L}, \tau_{\L})$ is still quasi-trivial (see Remark \ref{Ltorus}). Therefore, $\H^1_{cont}( \Gamma, G) = \{ 1 \}$.
\end{proof}

\begin{remark}
In the $\C/\R$ case, this simplification is always possible when the acting torus is quasi-trivial, that is it has no $\s^1$-factors (\cite[Lemma 4.12]{PA}). Recall that any real torus is isomorphic to a torus of the form $\GR^{n_0} \times (\s^1)^{n_1} \times \RC^{n_2}(\GC)$, with $n_0, n_1, n_2 \in \N$ \cite[Proposition 1.5]{Moser}. 
\end{remark}

\subsection{Split $\Bbbk$-torus actions} \label{SectionSplit}

A consequence of Hilbert's 90 Theorem (see \cite[Proposition III.8.24]{Berhuy}) on Theorem \ref{Result}, and on Proposition \ref{PropQuasiTrivialAH}, is that the   Altmann-Hausen presentation  of  \cite{Alt} extends \emph{mutatis mutandis} to split torus action on normal affine varieties over an  arbitrary field. 

\begin{proposition}[{Split torus actions on normal affine varieties over arbitrary fields}] \label{splittorus}    Let $\Bbbk$ be a  field, and let $\T=\Spec(\Bbbk[M])$ be a split $\Bbbk$-torus.
\begin{enumerate}[leftmargin=0.75cm, label=(\roman*)]
\item Let $(\omega_N, Y, \D)$ be an AH-datum over $\Bbbk$. The affine scheme $X[Y,\D] := \Spec(A[Y, \D])$ is a normal $\Bbbk$-variety endowed with a $\T$-action.
\item Let $X$ be an affine normal $\Bbbk$-variety endowed with a $\T$-action. There exists an AH-datum $(\omega_N, Y, \D)$  over $\Bbbk$ such that there is an isomorphism of $\T$-varieties $X \cong X[Y, \D]$.
\end{enumerate}
\end{proposition}

\begin{proof}   Let $\tau:= \tau_0$ be the natural $\Bbbk$-structure  on $\T_{\Bbbkb} = \Spec(\Bbbkb[M])$ defined in Example \ref{ExSplit}. One can shows that in our setting, a generalized AH-datum $(\omega_N, Y', \D', \sigma_Y, h=1)$ over $\Bbbkb$ corresponds to an AH-datum $(\omega_N, Y, \D)$  over $\Bbbk$.
\end{proof}

Proposition \ref{splittorus} has a useful consequence on  Theorem \ref{Result}: we can replace the infinite Galois extension $\Bbbkb/\Bbbk$ by any finite Galois extension that splits the acting torus.

\begin{corollary}[Torus actions on normal affine varieties over arbitrary fields] \label{ResultSplit} 
Let $\Bbbk'/\Bbbk$ be a finite Galois extension of Galois group $\Gamma$. Let $(\T = \Spec(\Bbbk'[M]), \tau)$ be a $\Bbbk$-torus. 
\begin{enumerate}[leftmargin=0.75cm, label=(\roman*)]
\item Let $(\omega_N, Y, \D, \sigma_Y, h)$ be a generalized AH-datum over $\Bbbk'$.  Then the affine $\T$-variety $X[Y, \D]$ admits a $\Bbbk$-structure $\sigma_{X[Y, \D]}$ such that $(\T, \tau)$ acts on $(X[Y, \D], \sigma_{X[Y, \D]})$.
\item Let $(X, \sigma)$ be a normal affine $\Bbbk$-variety endowed with a $(\T, \tau)$-action. Then there exists a generalized AH-datum $(\omega_N, Y, \D, \sigma_Y, h)$ such that there is an isomorphism of $(\T, \tau)$-varieties $(X, \sigma) \cong ( X[Y, \D], \sigma_{X[Y, \D]})$.
\end{enumerate}
\end{corollary}

\begin{proof} 
In the proof of Theorem \ref{Result}, we use the Altmann-Hausen presentation over an algebraically closed field given in \cite{Alt} combined with a toric downgrading (see Proposition \ref{ToricDown}). By Proposition \ref{splittorus}, this presentation extends over any field. Hence, combining   Proposition \ref{splittorus} together with Proposition \ref{ToricDown}, we obtain the desired result.
\end{proof}

\begin{slogan}
By Corollary \ref{ResultSplit}, we obtain an effective method to compute an AH-datum of a $T$-action on a $\Bbbk$-variety $X$. We  consider  a finite Galois extension $\Bbbk'/\Bbbk$ that splits $T$. Then, we determine an AH-datum for the $T_{\Bbbk'}$-action on $X_{\Bbbk'}$ using a $\Gal(\Bbbk'/\Bbbk)$-equivariant embedding $X_{\Bbbk'}$ in some $\Akp^n$ (Proposition \ref{ToricDown}), and then we  deduce an AH-datum for the $T$ action on a $X$.
\end{slogan}
 
\begin{remark} \label{RmkLanglois}
Corollary \ref{ResultSplit} generalizes a result of Langlois (see \cite[Theorem 5.10]{Langlois} and \cite{LangloisCorrected}) that focuses on  $\Bbbk$-torus actions of  complexity one, where $\Bbbk$ is an arbitrary field. Indeed, the geometrico-combinatorial presentation described in \cite{Langlois} is  over a splitting field $\Bbbk'$  of the $\Bbbk$-torus (so $\Bbbk'/\Bbbk$ is a finite Galois extension).
\end{remark}

\section{Two-dimensional torus actions} \label{Section2Torus}

In this last section, we focus on two-dimensional torus actions on normal affine varieties.  Given an affine variety endowed with a two-dimensional torus action, we provide in Section \ref{Section2Torus1} a sufficient condition to have a presentation as in Lemma \ref{simpl} (see Example \ref{ExDim2}). 
In Section \ref{Section2Torus2}, using birational geometry tools, we relate the torsor described in Remark \ref{Ltorus} to certain Del Pezzo surface (see Proposition \ref{birgeom}). 

\subsection{Two-dimensional tori and finite subgroups of $\Gl_2(\Z)$} \label{Section2Torus1}
In this section, we have compiled results from the literature. We describe the set of two-dimensional tori in Proposition \ref{FiniteSubgroups} and we give a description of the Galois-cohomology set classifying $T$-torsors in Theorem \ref{ElizondoH1}: it is exactly the result we need to determine if the Altmann-Hausen presentation simplifies in the case of two-dimensional torus actions on normal affine varieties. 

Let $\Bbbk'/ \Bbbk$ be a non-necessarily finite Galois extension and denote by $\Gamma$ its Galois group. Let $(\Gkp^n, \tau)$ be a $\Bbbk$-torus. Recall that $\tau$ induces  a $\Gamma$-action  $\tilde{\tau}$ on the character lattice $M \cong \Z^n$   satisfying $\tilde{\tau}_{\gamma_1 \gamma_2} = \tilde{\tau}_{\gamma_1} \circ \tilde{\tau}_{\gamma_2}$, and a dual $\Gamma$-action   $\hat{\tau}$ on the cocharacter lattice $N \cong \Z^n$   satisfying $\hat{\tau}_{\gamma_1 \gamma_2} = \hat{\tau}_{\gamma_2} \circ \hat{\tau}_{\gamma_1}$ (see Remark \ref{representation}).

The image of a $\Gamma$-representation is a finite subgroup $\mathcal{G}$ of $\Gl_n(\Z)$ (see Appendix \ref{Appendix2Split}).    Moreover, if $\tau'$ is a $\Bbbk$-group structure on $\Gkp^n$ equivalent to $\tau$, then the associated finite subgroups of $\Gl_n(\Z)$ are conjugated. The converse is false (see Example \ref{IneqConjug}), except in dimension 1. Therefore, since $\Gl_1(\Z) = \{1, -1 \}$, a one-dimensional $\Bbbk$-torus  is, up to isomorphism, either a split $\Bbbk$-torus or a norm one $\Bbbk$-torus.

\begin{example}[{Inequivalent $\Gamma$-representation on $\Gl_2(\Z)$ having same image $\mathcal{G}_6$ in $\Gl_2(\Z)$}] \label{IneqConjug} 
Let $\Bbbk'/\Bbbk$ be a finite Galois extension of degree 4 such that $\Gal(\Bbbk'/\Bbbk) \cong \mathscr{C}_2 \times \mathscr{C}_2 = \{ (1;1) , (1; -1) ,$ $ (-1;1) , (-1;-1) \}$. Consider $\rho, \rho' $ be two representations of $\Gamma$ in $\Gl_{2}(\Z)$ defined by:
\begin{align*}
\rho : \  &\Gamma \to \Gl_{2}(\Z);  \
(1;1)   \mapsto I; \
(1;-1)  \mapsto s;  \ 
(-1;1)  \mapsto -s; \ 
(-1;-1) \mapsto -I; \\
\rho' : \   &\Gamma \to \Gl_{2}(\Z); \  
(1;1)   \mapsto I;  \
(1;-1)  \mapsto s;  \
(-1;1)  \mapsto -I; \
(-1;-1) \mapsto -s,  
\end{align*}
where $s$ is the matrix defined in Proposition \ref{FiniteSubgroups}. The representations $\rho$ and $\rho'$ are inequivalent. 
\end{example} 

In the next proposition, we give the classification of finite subgroups of $\Gl_2(\Z)$ up to conjugacy.

\begin{proposition}[{\cite[\S 1.10.1]{Lorenz}}] \label{FiniteSubgroups}
Let $\mathscr{D}_n$ (resp. $\mathscr{C}_n$) be the dihedral (resp. cyclic) group of order $n$, and let $\mathscr{S}_n$ be the group of permutations of a set with $n$ elements. Let 
$$d:= 
\begin{bmatrix}
-1 & 0 \\
0 & 1 \\
\end{bmatrix}; \ 
s:=
\begin{bmatrix}
0 & 1 \\
1 & 0 \\
\end{bmatrix}; \ and \ 
x:=
\begin{bmatrix}
1 & -1 \\
1 & 0 \\
\end{bmatrix}.
$$
The non-trivial conjugacy classes of finite subgroups of $\Gl_2(\Z)$ are: \hfill
\begin{multicols}{2}
\begin{adjustwidth}{-0.5cm}{-0.5cm}
\begin{flushleft}
\begin{tabular}{|c|c|c|c|}
  \hline
  Label & Order & Generators & Isomorphism type \\
  \hline
  $\mathcal{G}_{1}$ & 12 & $x, s$ &  $\mathscr{D}_{12} \cong \mathscr{S}_3 \times \mathscr{C}_2$ \\
  $\mathcal{G}_{2}$ & 8 & $d, s$ &  $\mathscr{D}_{8}$ \\
	$\mathcal{G}_{3}$ & 6 & $x^2, s$ & $\mathscr{D}_{6} \cong \mathscr{S}_3$ \\
	$\mathcal{G}_{4}$ & 6 & $x^2, -s$ & $\mathscr{D}_{6} \cong \mathscr{S}_3$ \\
	$\mathcal{G}_{5}$ & 4 & $d, -d$ & $\mathscr{C}_{2} \times \mathscr{C}_2$ \\
	$\mathcal{G}_{6}$ & 4 & $s, -s$ & $\mathscr{C}_{2} \times \mathscr{C}_2$ \\
		\hline 
\end{tabular}
\end{flushleft}
\columnbreak
\begin{flushright}
\begin{tabular}{|c|c|c|c|}
  \hline
  Label & Order & Generators & Isomorphism type \\
  \hline
	$\mathcal{G}_{7}$ & 6 & $x$ & $\mathscr{C}_{6}$  \\

	$\mathcal{G}_{8}$ & 4 & $ds$ & $\mathscr{C}_{4}$ \\
	$\mathcal{G}_{9}$ & 3 & $x^2$ & $\mathscr{C}_{3}$ \\
	$\mathcal{G}_{10}$ & 2 & $x^3=-id$ & $\mathscr{C}_{2}$ \\
	$\mathcal{G}_{11}$ & 2 & $d$ & $\mathscr{C}_{2}$ \\
	$\mathcal{G}_{12}$ & 2 & $s$ & $\mathscr{C}_{2}$ \\
  \hline
\end{tabular}
\end{flushright}
\end{adjustwidth}
\end{multicols}
\noindent The following diagrams represents the inclusions of normal subgroups (the number is the index): \hfill
\vspace{-0.25cm}
\small
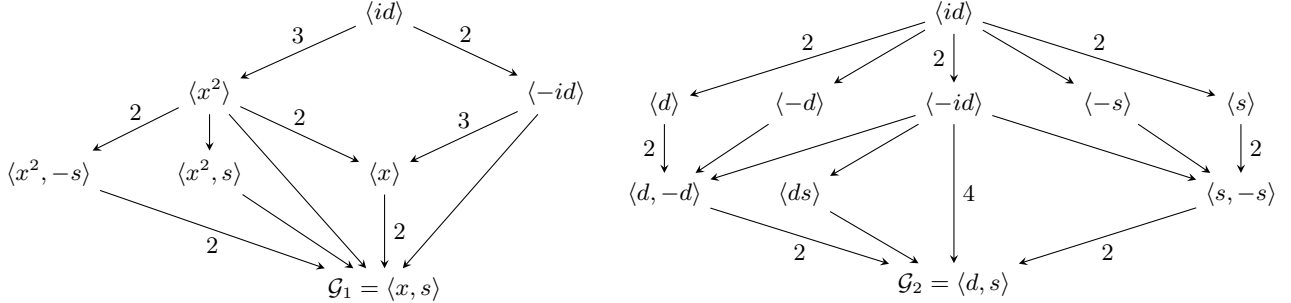
\begin{multicols}{2}
\begin{adjustwidth}{-1cm}{-1.25cm}
\begin{flushleft}
\begin{tikzpicture}
  \matrix (m) [matrix of math nodes,row sep=1.5em,column sep=2.75em,minimum width=2em]
  {
  & &   \langle id \rangle &    \\
  & \langle x^2 \rangle & & \langle -id \rangle  \\
 \langle x^2, -s \rangle &  \langle x^2, s \rangle & \langle x \rangle   & \\
 & & & & \\
  & & \mathcal{G}_{1} = \langle x,s \rangle  &   \\
};
  \path[-stealth]
    (m-1-3) edge node [above] {$3$} (m-2-2)
            edge node [above] {$2$} (m-2-4)
    (m-2-2) edge node [above] {$2$} (m-3-1)
		        edge  (m-3-2)
						edge node [above] {$2$} (m-3-3)
						edge  (m-5-3)
		(m-2-4) edge node [above] {$3$} (m-3-3)
		        edge  (m-5-3)
		(m-3-1) edge node [below] {$2$} (m-5-3)
		(m-3-2) edge  (m-5-3)
		(m-3-3) edge node [right] {$2$} (m-5-3);
\end{tikzpicture}

\end{flushleft}
\columnbreak
\begin{flushright}
\begin{tikzpicture}
  \matrix (m) [matrix of math nodes,row sep=2em,column sep=2.25em,minimum width=2em]
  {
 & & \langle id \rangle & &  \\
 \langle d \rangle & \langle -d \rangle & \langle -id  \rangle &  \langle -s \rangle & \langle s \rangle\\
\langle d, -d \rangle & \langle ds \rangle & &  &   \langle s, -s \rangle \\
   & & \mathcal{G}_{2} = \langle d,s \rangle & &  \\
};
  \path[-stealth]
    (m-1-3) edge node [above] {$2$} (m-2-1)
		        edge   (m-2-2)
            edge node [left] {$2$} (m-2-3)
						edge   (m-2-4)
						edge node [above] {$2$} (m-2-5)
    (m-2-1) edge node [left] {$2$} (m-3-1)
		(m-2-2) edge   (m-3-1)
		(m-2-3) edge   (m-3-1)
		        edge   (m-3-2)
						edge node [right] {$4$} (m-4-3)
						edge   (m-3-5)
		(m-2-4) edge   (m-3-5)
		(m-2-5) edge node [right] {$2$} (m-3-5)
		(m-3-1) edge node [below] {$2$} (m-4-3)
		(m-3-2) edge   (m-4-3)
		(m-3-5) edge node [below] {$2$} (m-4-3);
\end{tikzpicture}
\end{flushright}
\end{adjustwidth}
\end{multicols}
\normalsize
\end{proposition}

\begin{example} \label{Extensions} Using inverse Galois theory, we can realize each of the  conjugacy classes $\mathcal{G}_i$ by a $\Q$-torus that splits over a finite Galois extension $\Bbbk'/\Bbbk$,  such that $\Gal(\Bbbk'/\Bbbk) \cong \mathcal{G}_i$. \hfill
\small
\begin{adjustwidth}{-0.25cm}{-0.25cm}
\begin{multicols}{2}
\begin{flushleft}
\begin{tabular}{|c|c|}
  \hline
  Label & Galois extension of $\Q$ of Galois group $\mathcal{G}_i$ \\
  \hline
  $\mathcal{G}_{1}$ & $\Q \left( \sqrt[3]{1+\sqrt{2}} \right)$ \cite[Exercise 8.9]{Escofier} \\
  $\mathcal{G}_{2}$ &  $\Q \left( \sqrt[4]{2},i \right)$ \cite[\S 8.7.3]{Escofier}  \\
	$\mathcal{G}_{3}$ &  $\Q \left( \sqrt[3]{2}, j \right)$ \cite[\S 8.3]{Escofier}  \\
	$\mathcal{G}_{4}$ &  $\Q \left( \sqrt[3]{2}, j \right)$  \\
	$\mathcal{G}_{5}$ & $\Q(\sqrt{2}, \sqrt{3})$ (see Example \ref{ExampleBiquadratic})   \\
	$\mathcal{G}_{6}$ &  $\Q(\sqrt{2}, \sqrt{3})$  \\
		\hline 
\end{tabular}
\end{flushleft}
\columnbreak
\begin{flushright}
\begin{tabular}{|c|c|}
  \hline
  Label & Galois extension of $\Q$ of Galois group $\mathcal{G}$  \\
  \hline
	$\mathcal{G}_{7}$ &  $\Q \left( e^{\frac{2i \pi}{7}} \right)$ \cite[\S 9.5]{Escofier} \\
	$\mathcal{G}_{8}$ &  $\Q \left( e^{\frac{2i \pi}{5}} \right)$ \cite[\S 9.5]{Escofier}  \\
	$\mathcal{G}_{9}$ &  $\Q \left( \cos \left( \frac{2 \pi}{7} \right) \right)$ 
		\\
	$\mathcal{G}_{10}$ & $\Q(i)$   \\
	$\mathcal{G}_{11}$ & $\Q(i)$  \\
	$\mathcal{G}_{12}$ & $\Q(i)$  \\
  \hline
\end{tabular}
\end{flushright}
\end{multicols}
\end{adjustwidth}
\normalsize
\end{example}

Based on Proposition \ref{FiniteSubgroups}, a  classification of two-dimensional $\Bbbk$-tori is obtained in \cite{VosI}. 

\begin{proposition}[{\cite{VosI}}] \label{Qtori} 
Let $T$ be a two-dimensional $\Bbbk$-torus, and let $\tau$ be the corresponding $\Bbbk$-group structure on $T_{\Bbbkb}$. Let $\Bbbk'$ be a splitting field of $T$ such that $\Gal(\Bbbk'/\Bbbk) \cong  \mathcal{G}$, where $\mathcal{G}$ is the image of $\tilde{\tau}$. Then, $T$ is isomorphic to one of the following tori. \hfill
\smallskip
\begin{adjustwidth}{-1.2cm}{-1.2cm}
\small
\begin{center}
\begin{tabular}{|c|l|}
  \hline
 $\Gal(\Bbbk'/\Bbbk)$ & Corresponding $\Bbbk$-tori \\
  \hline
 $\mathcal{G}_{1} = \langle x,s \rangle$ &      $\text{R}_{\Bbbk_1/\Bbbk} \left( \text{R}^{(1)}_{\Bbbk'/\Bbbk_1} \left( \Gkp \right) \right) \cap  \text{R}_{\Bbbk_2/\Bbbk} \left( \text{R}^{(1)}_{\Bbbk_3/\Bbbk_2} \left( \Gkth \right) \right)^{\vphantom{d}}$,   where $\Bbbk_1=\Bbbk'^{\langle x^2 \rangle}$ is a Galois extension  of $\Bbbk$ of \\ &  degree 4, where $\Bbbk_2=\Bbbk'^{\langle x^is, - x^is \rangle}$,   and where  $\Bbbk_3=\Bbbk'^{\langle x^is \rangle}$ for $i \in \{ 0, \dots ,5 \}$ $(\Bbbk_3/\Bbbk_2$ is Galois).   \\ 
  $\mathcal{G}_{2}= \langle d,s \rangle$   &  $\text{R}_{\Bbbk_2/\Bbbk} \left( \text{R}^{(1)}_{\Bbbk_1/\Bbbk_2} \left( \Gko \right) \right)$, where $\Bbbk_1=\Bbbk'^{\langle d \rangle}$ and $\Bbbk_2=\Bbbk'^{\langle d, -d \rangle}$, or $\Bbbk_1=\Bbbk'^{\langle s \rangle}$ and $\Bbbk_2=\Bbbk'^{\langle s, -s \rangle}$.   \\
	$\mathcal{G}_{3}= \langle x^2, s \rangle$  & $\text{R}_{\Bbbk_1/\Bbbk} \left( \text{R}^{(1)}_{\Bbbk'/\Bbbk_1} \left( \Gkp \right) \right) \cap \text{R}_{\Bbbk_2/\Bbbk} \left(\Gkt \right)$, where $\Bbbk_1=\Bbbk'^{\langle x^2 \rangle}$, and where $\Bbbk_2$  goes through the cubic \\ & subfields $\Bbbk'^{\langle s \rangle}$, $\Bbbk'^{\langle x^2 s \rangle}$, and $\Bbbk'^{\langle x^4 s \rangle}$.   \\
	$\mathcal{G}_{4}= \langle x^2, -s \rangle$  &  $\text{R}_{\Bbbk_1/\Bbbk} \left( \text{R}^{(1)}_{\Bbbk'/\Bbbk_1} \left( \Gkp \right) \right) \cap \text{R}_{\Bbbk_2/\Bbbk} \left( \text{R}^{(1)}_{\Bbbk'/\Bbbk_2} \left( \Gkp \right) \right)$, where $\Bbbk_1=\Bbbk'^{\langle x^2 \rangle}$, and where $\Bbbk_2$  goes through the \\ & cubic subfields $\Bbbk'^{\langle -s \rangle}$, $\Bbbk'^{\langle -x^2 s \rangle}$, and $\Bbbk'^{\langle -x^4 s \rangle}$.   \\ 
	$\mathcal{G}_{5}= \langle d, -d \rangle$ & $\text{R}^{(1)}_{\Bbbk_1/\Bbbk} \left( \Gko \right) \times \text{R}^{(1)}_{\Bbbk_2/\Bbbk} \left( \Gkt \right)$,   where  $\Bbbk_1 \Bbbk_2= \Bbbk'$.  \\
	 $\mathcal{G}_{6}= \langle s, -s \rangle$ &   $\text{R}_{\Bbbk_1/\Bbbk} \left(     \text{R}^{(1)}_{\Bbbk'/\Bbbk_1} \left( \Gkp \right) \right)$,   where   $\Bbbk_1$ goes through the Galois extensions $\Bbbk'^{\langle s \rangle}$, $\Bbbk'^{\langle -s \rangle}$, and $\Bbbk'^{\langle -id \rangle}$ of $\Bbbk$.  \\
$\mathcal{G}_{7} = \langle x \rangle$ & $\text{R}_{\Bbbk_2/\Bbbk} \left(     \text{R}^{(1)}_{\Bbbk'/\Bbbk_2} \left( \Gkp \right) \right) \cap \text{R}_{\Bbbk_1/\Bbbk} \left(     \text{R}^{(1)}_{\Bbbk'/\Bbbk_1} \left( \Gkp \right) \right)$, where $\Bbbk_1=\Bbbk'^{\langle x^2 \rangle}$ and $\Bbbk_2 = \Bbbk'^{\langle -id \rangle}$.   \\
$\mathcal{G}_{8}= \langle ds \rangle$  &  $\text{R}_{\Bbbk_1/\Bbbk} \left( \text{R}^{(1)}_{\Bbbk'/\Bbbk_1} \left( \Gkp \right) \right)$,  where   $\Bbbk_1 = \Bbbk'^{\langle -id \rangle}$. \\
$\mathcal{G}_{9}= \langle x^2 \rangle$ &  $\text{R}^{(1)}_{\Bbbk'/\Bbbk} \left( \Gkp \right)$.    \\
$\mathcal{G}_{10}= \langle  -id \rangle$ &  $\text{R}^{(1)}_{\Bbbk'/\Bbbk} \left( \Gkp \right) \times \text{R}^{(1)}_{\Bbbk'/\Bbbk} \left( \Gkp \right)$.  \\
$\mathcal{G}_{11}= \langle d \rangle$ &  $\text{R}^{(1)}_{\Bbbk'/\Bbbk} \left( \Gkp \right) \times \Gk$. \\
$\mathcal{G}_{12}= \langle s \rangle$	 &  $\text{R}_{\Bbbk'/\Bbbk} \left( \Gkp \right)$.  \\
  \hline
\end{tabular}
\end{center}
\normalsize
\end{adjustwidth}
\end{proposition}
\smallskip

We have described the two-dimensional tori, now we describe their torsors. 
In \cite[\S 5]{Elizondo}, there is a complete description of the Galois-cohomology set used to  classify   $T$-torsors, where $T$ is a two-dimensional $\Bbbk$-torus. To each Galois extension $\Bbbk'/\Bbbk$ such that $\Gal(\Bbbk'/\Bbbk) \cong \mathcal{G}$ is a finite subgroup   of $\Gl_2(\Z)$ (see {\cite[Proposition 3.7]{Elizondo}),   an expression of $H^1(\mathcal{G}, \Aut^{T_{\Bbbk'}}(T_{\Bbbk'}))$ depending on a Brauer group  is obtained. In order to state  results of \cite[\S 5]{Elizondo},  we briefly introduce  some notations.

Let $(\T=\Spec(\Bbbk'[M]), \tau)$ be a two-dimensional $\Bbbk$-torus. By {\cite[Proposition 3.7]{Elizondo}, we can assume that  $\Bbbk'/\Bbbk$ is a finite Galois extension such that the $\Gamma$-representation $\tilde{\tau} : \Gamma \to \Gl(M) \cong \Gl_2(\Z)$ is injective. Denote its image by $\mathcal{G} \cong  \Gamma = \Gal(\Bbbk'/\Bbbk)$. In \cite[\S 5]{Elizondo}, to compute  $H^1(\mathcal{G}, \Aut^{T_{\Bbbk'}}(T_{\Bbbk'}))$,  the authors distinguish if $\mathcal{G}$ has an element of order 3 or not. 
If $\mathcal{G}$ does not have an element of order 3, then, up to conjugacy,  $\mathcal{G}$ is a subgroup of $\mathcal{G}_{2}$.
If $\mathcal{G}$ has an element of order 3, then, up to conjugacy,  $\mathcal{G}$ is one of the following subgroup of $\mathcal{G}_{1}$: $\mathcal{G}_{1}$,  $\mathcal{G}_{3}$, $\mathcal{G}_{4}$, $\mathcal{G}_{7}$ or $\mathcal{G}_{9}$. In this context, we obtain subextensions of $\Bbbk'/\Bbbk$
\vspace{-0.25cm}
\small 
\begin{center}
\begin{tikzpicture}
  \matrix (m) [matrix of math nodes,row sep=0.01cm,column sep=1cm,minimum width=0.01cm]
  {
       & \Bbbk_1 &         &        \\
 \Bbbk &         & \Bbbk_3 & \Bbbk', \\
       & \Bbbk_2 &         &        \\};
  \path[-stealth]
    (m-2-1) edge   (m-1-2)
		        edge   (m-3-2)
		(m-1-2) edge   (m-2-3)
    (m-3-2) edge   (m-2-3)
		(m-2-3) edge   (m-2-4);
\end{tikzpicture}
\end{center}
\normalsize
where $\Bbbk_1:= {\Bbbk'}^{\mathcal{G} \cap \mathcal{G}_6}$, $\Bbbk_2:= {\Bbbk'}^{\mathcal{G} \cap \mathcal{G}_3}$,  and $\Bbbk_3:= {\Bbbk'}^{\mathcal{G} \cap \mathcal{G}_6 \cap \mathcal{G}_3}$.
Consider the homomorphism  $\beta : \mathrm{Br}(\Bbbk_1/\Bbbk) \to \mathrm{Br}(\Bbbk_3/\Bbbk_2)$ obtained by base extension from $\Bbbk$ to $\Bbbk_2$. 
Consider the homomorphism  $\eta : \mathrm{Br}(\Bbbk_3/\Bbbk_1) \to \mathrm{Br}(\Bbbk_2/\Bbbk)$ induced by the norm map $N_{\Bbbk_1/\Bbbk}(\Bbbk) : \Bbbk_1^* \to \Bbbk^*$ (see Definition \ref{DefQuasiTrivialNorm}). We denote the kernel of $\eta$ by $\mathrm{Br}_{\eta}(\Bbbk_1/\Bbbk |\Bbbk_3/\Bbbk_2)$. 
In Theorem \ref{ElizondoH1}, we summarize main results of \cite{Elizondo}. Compare with Theorem \ref{Qtori} combined with Lemma \ref{H1Norm1}}.

\begin{theorem}[{\cite[Theorems 5.3 \& 5.5]{Elizondo}}]  \label{ElizondoH1} With the previous notations, we have \hfill 
\small
\begin{multicols}{2}
\begin{adjustwidth}{-0.75cm}{-0.75cm} 
\begin{flushleft}
\begin{tabular}{|c|c|}
  \hline
  $\mathcal{G}$ &  $H^1 \left( \mathcal{G}, \Aut^{\T}(\T) \right)^{\vphantom{d}}$ \\
  \hline
  $\langle x,s \rangle$ &  $\frac{\mathrm{Br} \left( {\Bbbk'}^{\langle s \rangle}/{\Bbbk'}^{\langle x^2, s \rangle} \right)^{\vphantom{d}}}{\beta \left( \mathrm{Br}\left({\Bbbk'}^{\langle s,-s \rangle}/\Bbbk\right) \right)} \oplus \mathrm{Br}_{\eta} \left({\Bbbk'}^{\langle s,-s \rangle}/\Bbbk |{\Bbbk'}^{\langle s \rangle}/{\Bbbk'}^{\langle x^2, s \rangle}\right)$   \\
  $\langle d,s \rangle$ &   $\mathrm{Br} \left(  {\Bbbk'}^{\langle d \rangle} / {\Bbbk'}^{\langle d, -d \rangle} \right)$  \\
	$\langle x^2, s \rangle$ & $\mathrm{Br}\left({\Bbbk'}^{\langle s \rangle}/\Bbbk \right)$     \\
	$\langle x^2, -s \rangle$ &  $\frac{\mathrm{Br} \left(\Bbbk'/{\Bbbk'}^{\langle x^2 \rangle} \right)^{\vphantom{d}}}{\beta \left(\mathrm{Br}({\Bbbk'}^{\langle -s \rangle}/\Bbbk) \right)}$    \\
	$\langle d, -d \rangle$ &  $\mathrm{Br} \left( {\Bbbk'}^{\langle d \rangle}/{\Bbbk} \right) \oplus  \mathrm{Br} \left(  {\Bbbk'}^{\langle d \rangle}/ {\Bbbk}\right)^{\vphantom{d}}$     \\
		\hline 
\end{tabular}
\end{flushleft}
\columnbreak
\begin{flushright}
\vspace{-0.25cm}
\begin{tabular}{|c|c|}
  \hline
  $\mathcal{G}$ &  $H^1 \left( \mathcal{G}, \Aut^{\T}(\T) \right)^{\vphantom{d}}$ \\
  \hline
	$\langle s, -s \rangle$ &   $\mathrm{Br} \left( {\Bbbk'} / {\Bbbk'}^{\langle -id \rangle} \right)^{\vphantom{d}}$    \\
  $\langle x \rangle$ &   $\frac{\mathrm{Br} \left( \Bbbk'/{\Bbbk'}^{\langle x^2 \rangle} \right)}{\beta \left( \mathrm{Br}(\Bbbk'/\Bbbk) \right)} \oplus \mathrm{Br}_{\eta} \left(\Bbbk'/\Bbbk |\Bbbk'/{\Bbbk'}^{\langle x^2 \rangle} \right)$ \\
  $\langle ds \rangle$ &  $\mathrm{Br} \left( \Bbbk'/ {\Bbbk'}^{\langle -id \rangle} \right)$   \\
	$\langle x^2 \rangle$ &  $\mathrm{Br}(\Bbbk'/\Bbbk)^{\vphantom{d}}$    \\
	$\langle -id \rangle$ &   $\mathrm{Br}(\Bbbk'/\Bbbk) \oplus \mathrm{Br}(\Bbbk'/\Bbbk)$  \\
	$\langle d \rangle$ &   $\mathrm{Br}(\Bbbk'/\Bbbk)$  \\
	$\langle s \rangle$ &  $\{ 0 \}$   \\
		\hline 
\end{tabular}
\end{flushright}
\end{adjustwidth}
\end{multicols}
\normalsize
\end{theorem}

\begin{example}
Let   $(\T=\Spec(\Bbbkb[M]), \tau)$ be a two-dimensional $\Bbbk$-torus such that the image of $\tilde{\tau}$  is conjugated to $\mathcal{G}_{12}$; for instance $\text{R}_{\C/\R}(\GC)$.  Then, $H^1(\mathcal{G}, \Aut^{\T}(\T)) =\{ 1 \}$. Therefore, if $(\T, \tau)$ acts on an affine normal variety $(X, \sigma)$, then the Altmann-Hausen presentation simplifies (see Lemma \ref{simpl}). 
\end{example}

\begin{example}  \label{ExDim2}
Let $(\T=\Spec(\Bbbk'[M]), \tau)$ be a two-dimensional $\Bbbk$-torus such that  the image $\mathcal{G}$ of $\tilde{\tau}$  is isomorphic to $\Gamma$. 
Let $(\omega_N, Y, \D, \sigma_Y, h)$ be a generalized AH-datum.  Let $\L:= \Bbbk'(Y)$ and let $\K := \L^{\Gamma}$. By Lemma \ref{Artin}, $\L/\K$ is a Galois extension of Galois group $\Gamma$. If $\K$ satisfies one of the conditions mentioned in \cite[Chapter X, \S 7]{SerreLocal}, and  if $\mathcal{G}$ is conjugated to $\mathcal{G}_i$ for $i \in \{ 3, 5, 9, 10, 11 \}$, then the Altmann-Hausen presentation simplifies (see Lemma \ref{simpl}), as in Example \ref{ExNorm1deg2'}.
\end{example}

\subsection{Galois-equivariant MMP} \label{Section2Torus2}

In this section, $\Bbbk'=\Bbbkb$ and $\Gamma=\Gal(\Bbbkb/\Bbbk)$.

A smooth projective $\Bbbk$-surface $(X, \sigma)$ is minimal if and only if $X$ admits no Galois stable set of pairwise disjoint $(-1)$-curves (see \cite[Theorem 21.8 p115]{ManinCubic}, see also \cite{KollarMori} for a point on birational geometry).

Recall that a smooth projective $\Bbbk$-surface $V$ is called a Del Pezzo surface if the dual of its canonical sheaf is ample. 
The degree of $V$ is the degree of $V_{\Bbbkb}$, that is the self-intersection number of any anti-canonical Weil divisor $-K_{V_{\Bbbkb}}$ on $V_{\Bbbkb}$. Recall  that $1 \leq d \leq 9$. 
Over an algebraically (or separably) closed field, Del Pezzo surfaces are classified, up to deformation, by their degree (see \cite[Theorems 24.3 and 24.4]{ManinCubic}). Up to deformation, the Del Pezzo surfaces are: $\Pkb^1 \times \Pkb^1$ (if $d=8$),  and; the blowup of $\Pkb^2$ at $r=9-d$ $\Bbbkb$-points in general position. For a classification over an arbitrary field, see \cite{ManinCubic}, \cite{DolgAlgGeom}.

\begin{remark}
For a complement on Del Pezzo surfaces, see for instance \cite{Wall}, or \cite{Serganova}. In  \cite{Russo}, there is a complete classification of real Del Pezzo surfaces. 
\end{remark}

Compactification of torsors is studied in \cite[Section 1, Theorem 1 and its Corollary]{Vos83}, and also in \cite[Theorem 3.9]{Elizondo}. In this article we need the following lemma.

\begin{lemma}  \label{Compactification} 
Let $(\T, \sigma)$ be a $(\T, \tau)$-torsor, where $\T=\Spec(\Bbbkb[M])$ and $M \cong \Z^2$. Then $(\T, \sigma)$ admits a smooth projective compactification that is a $(\T, \tau)$-toric Del Pezzo $\Bbbk$-surface containing $(\T, \sigma)$ as a $(\T, \tau)$-stable dense open subset.
\end{lemma}

\begin{proof}
The $\Bbbk$-group structure $\tau$ induces a $\Gamma$-representation $\hat{\tau}$ on $\Gl(N)$.  The image of $\hat{\tau}$ is a finite subgroup of $\Gl_2(\Z)$.  

$\bullet$ \emph{Case 1: } The image of $\hat{\tau}$ is conjugated to a normal subgroup of  $\mathcal{G}_1 = \langle x, s \rangle$ (see Proposition \ref{FiniteSubgroups}). 
Let $\Lambda$ be the  fan of the projective toric variety $\text{Bl}_3(\Pkb^2)$ in $N_{\Q}$:  
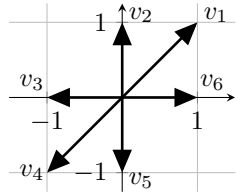
\begin{center}
\begin{tikzpicture}[line cap=round,line join=round,>=triangle 45,x=0.7cm,y=0.7cm]
\begin{axis}[
x=1cm,y=01cm,
axis lines=middle,
ymajorgrids=true,
xmajorgrids=true,
xmin=-1.5,
xmax=1.5,
ymin=-1.25,
ymax=1.25,
xtick={-3.0,-2.0,...,3.0},
ytick={-3.0,-2.0,...,3.0},]
\clip(-3.,-3.) rectangle (3.,3.);
\draw [->,line width=1.pt] (0.,0.) -- (1.,0.);
\draw [->,line width=1.pt] (0.,0.) -- (0.,1.);
\draw [->,line width=1.pt] (0.,0.) -- (-1.,-1.);
\draw [->,line width=1.pt] (0.,0.) -- (1.,1.);
\draw [->,line width=1.pt] (0.,0.) -- (-1.,0.);
\draw [->,line width=1.pt] (0.,0.) -- (0.,-1.);
\begin{scriptsize}
\draw[color=black] (1.25,1.1) node {$v_1$};
\draw[color=black] (0.25,1.1) node {$v_2$};
\draw[color=black] (-1.2,0.2) node {$v_3$};
\draw[color=black] (-1.2,-1) node {$v_4$};
\draw[color=black] (0.25,-1.1) node {$v_5$};
\draw[color=black] (1.2,0.2) node {$v_6$};
\end{scriptsize}
\end{axis}
\end{tikzpicture}
\end{center}
This fan is $\Gamma$-stable (for $\hat{\tau}$), therefore, by \cite[Proposition 1.19]{Huruguen}, the $\Bbbk$-structure $\sigma$ on $\T$ extends to a $\Bbbk$-structure $\sigma_{\Lambda}$ on $X_{\Lambda} = \text{Bl}_3(\Pkb^2)$, the variety $(X_{\Lambda}, \sigma_{\Lambda})$ is a $(\T, \tau)$-toric $\Bbbk$-variety (a toric Del Pezzo $\Bbbk$-surface of degree 6), and we obtain a $(\T, \tau)$-equivariant open immersion $(\T, \sigma) \hookrightarrow (X_{\Lambda}, \sigma_{\Lambda})$. 

$\bullet$ \emph{Case 2: } The image of $\hat{\tau}$ is conjugated to a normal subgroup of  $\mathcal{G}_2 = \langle d, s \rangle$ (see Proposition \ref{FiniteSubgroups}). 
Let $\Lambda$ be the  fan of the projective toric variety $\Pkb^1 \times \Pkb^1$ in $N_{\Q}$:  
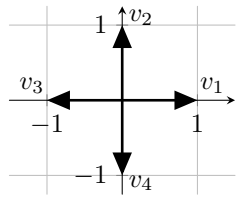
\begin{center}
\begin{tikzpicture}[line cap=round,line join=round,>=triangle 45,x=0.7cm,y=0.7cm]
\begin{axis}[
x=1cm,y=1cm,
axis lines=middle,
ymajorgrids=true,
xmajorgrids=true,
xmin=-1.5,
xmax=1.5,
ymin=-1.25,
ymax=1.25,
xtick={-3.0,-2.0,...,3.0},
ytick={-3.0,-2.0,...,3.0},]
\clip(-3.,-3.) rectangle (3.,3.);
\draw [->,line width=1.pt] (0.,0.) -- (1.,0.);
\draw [->,line width=1.pt] (0.,0.) -- (0.,1.);
\draw [->,line width=1.pt] (0.,0.) -- (-1.,0.);
\draw [->,line width=1.pt] (0.,0.) -- (0.,-1.);
\begin{scriptsize}
\draw[color=black] (1.2,0.2) node {$v_1$};
\draw[color=black] (0.25,1.1) node {$v_2$};
\draw[color=black] (-1.2,0.2) node {$v_3$};
\draw[color=black] (0.25,-1.1) node {$v_4$};
\end{scriptsize}
\end{axis}
\end{tikzpicture}
\end{center}
This fan is $\Gamma$-stable (for $\hat{\tau}$), therefore, by \cite[Proposition 1.19]{Huruguen}, the $\Bbbk$-structure $\sigma$ on $\T$ extends to a $\Bbbk$-structure $\sigma_{\Lambda}$ on $X_{\Lambda} = \Pkb^1 \times \Pkb^1 $, the variety $(X_{\Lambda}, \sigma_{\Lambda})$ is a $(\T, \tau)$-toric $\Bbbk$-variety  (a toric Del Pezzo $\Bbbk$-surface of degree 8),  and we obtain a $(\T, \tau)$-equivariant  open immersion $(\T, \sigma) \hookrightarrow (X_{\Lambda}, \sigma_{\Lambda})$. 
\end{proof}

Note that Lemma  \ref{Compactification} is still true for any field extension $\Bbbk'/\Bbbk$. We assume $\Bbbk'=\Bbbkb$ in view of applying an MMP.

\begin{proposition} \label{birgeom}
Let  $(\T, \tau)$ be a two-dimensional $\Bbbk$-torus and let $(\T, \sigma)$ be a  $(\T, \tau)$-torsor.    
Taking a smooth projective compactification, and then applying  a $(\T, \sigma)$-equivariant Minimal Model Program,  we obtain   a smooth toric Del Pezzo $\Bbbk$-surface $X$, birational to $(\T, \sigma)$,  that contains $(\T, \sigma)$  as a $(\T, \tau)$-stable dense open subset. More precisely, we have the following possibilities for $X$.
\begin{enumerate}[leftmargin=0.75cm, label=(\roman*)]
\item $X$ is a $\Bbbk$-form of $\Pkb^2$.  In this case, the image of $\hat{\tau}$ is conjugated to $\mathcal{G}_{4}$, $\mathcal{G}_{9}$, $\mathcal{G}_{12}$, or to $\langle id \rangle$.
\item $X$ is a $\Bbbk$-form of Picard rank one of the Del Pezzo $\Bbbkb$-surface of degree $6$. In this case, the image of $\hat{\tau}$ is conjugated to $\mathcal{G}_{1}$, $\mathcal{G}_{3}$, or to $\mathcal{G}_{7}$.
\item $X$ is a $\Bbbk$-form of Picard rank one of $\Pkb^1 \times \Pkb^1$. In this case, the image of $\hat{\tau}$ is conjugated to $\mathcal{G}_{2}$, $\mathcal{G}_{6}$, $\mathcal{G}_{8}$, or to $\mathcal{G}_{12}$. Or,  $X$ is a $\Bbbk$-form of Picard rank two of $\Pkb^1 \times \Pkb^1$ (if $(\T, \tau)$ is reducible). In this case, the image of $\hat{\tau}$ is conjugated to $\mathcal{G}_{5}$, $\mathcal{G}_{10}$, $\mathcal{G}_{11}$, or to $\langle id \rangle$.
\end{enumerate}
\end{proposition}

\begin{proof} By Lemma \ref{Compactification}, a compactification of $(\T, \sigma)$ is a $(\T, \tau)$-toric Del Pezzo surface $(X_{\Lambda}, \sigma_{\Lambda})$. We can assume that $X_{\Lambda} = \text{Bl}_3(\Pkb^2)$  or  $X_{\Lambda} = \Pkb^1 \times \Pkb^1$ (see the Proof of Lemma \ref{Compactification}). Recall that to a ray $v_i$ of a fan of a toric variety corresponds  a toric divisor $D_i$. 
The toric divisors associated to $X_{\Lambda} = \text{Bl}_3(\Pkb^2)$ and to $X_{\Lambda} = \Pkb^1 \times \Pkb^1$ are respectively:
\begin{multicols}{2}
\begin{tikzpicture}[line cap=round,line join=round,>=triangle 45,x=0.7cm,y=0.7cm]
\clip(-2,-1.25) rectangle (2,1.25);
\draw [-,line width=1.pt] (1.,0.) -- (0.5,0.866);
\draw [-,line width=1.pt] (0.5,0.866) -- (-0.5,0.866);
\draw [-,line width=1.pt] (-0.5,0.866) -- (-1.,0.);
\draw [-,line width=1.pt] (-1.,0.) -- (-0.5,-0.866);
\draw [-,line width=1.pt] (-0.5,-0.866) -- (0.5,-0.866);
\draw [-,line width=1.pt] (0.5,-0.866) -- (1.,0.);
\begin{scriptsize}
\draw[color=black] (1.1,0.5) node {$D_1$};
\draw[color=black] (0.2,1.1) node {$D_2$};
\draw[color=black] (-1.1,0.5) node {$D_3$};
\draw[color=black] (-1.1,-0.5) node {$D_4$};
\draw[color=black] (0.2,-1.1) node {$D_5$};
\draw[color=black] (1.1,-0.5) node {$D_6$};
\end{scriptsize}
\end{tikzpicture}

\begin{tikzpicture}[line cap=round,line join=round,>=triangle 45,x=0.7cm,y=0.7cm]
\clip(-2,-1.25) rectangle (2,1.25);
\draw [-,line width=1.pt] (1.,-1.) -- (1.,1.);
\draw [-,line width=1.pt] (1.,1.) -- (-1.,1.);
\draw [-,line width=1.pt] (-1.,1.) -- (-1.,-1.);
\draw [-,line width=1.pt] (-1.,-1.) -- (1.,-1.);
\begin{scriptsize}
\draw[color=black] (1.3,0) node {$D_1$};
\draw[color=black] (0,0.8) node {$D_2$};
\draw[color=black] (-1.3,0) node {$D_3$};
\draw[color=black] (0,-0.8) node {$D_4$};
\end{scriptsize}
\end{tikzpicture}
\end{multicols}
By construction, $\T = X_{\Lambda} \backslash \cup D_i$ and $(\T, \sigma) = (X_{\Lambda} \backslash \cup D_i, \sigma_{\Lambda})$ \cite[Theorem 30.3.1 p166]{ManinCubic}  \cite[Theorem 3.10 p108]{ManinRat1}.

Recall that $\Pkb^2$ has Picard rank one. Therefore, a $\Bbbk$-form of $\Pkp^2$ has also Picard rank one. 
Then, recall that $\text{Bl}_3(\Pkb^2)$ has Picard rank four and that the Picard group is generated by three disjoint toric divisors and a line that  does not intersect any of the six toric divisors. Therefore, a $\Bbbk$-form of $ \text{Bl}_3(\Pkb^2)$  has   Picard rank between one and four.
Finally, recall that $\Pkb^1 \times \Pkb^1$ has Picard rank two and that the Picard group is generated by two intersecting toric divisors. Therefore,  a $\Bbbk$-form of $\Pkb^1 \times \Pkb^1$  has   Picard rank one or two. 

We apply a  $(\T, \tau)$-equivariant MMP to the $\Bbbk$-surface $(X_{\Lambda}, \sigma_{\Lambda})$. This algorithm consists of contracting (-1)-curves in a way compatible with the $\Gamma$-action. 

Since the case where $X_{\Lambda} =  \Pkb^1 \times \Pkb^1$ is similar to the case where $X_{\Lambda} =  \text{Bl}_3(\Pkb^2)$, we only focus on the second one. So, let  $X_{\Lambda} =  \text{Bl}_3(\Pkb^2)$, it has six   (-1)-curves. 

$\bullet$ If the image of $\hat{\tau}$ is conjugated to $\langle x, s \rangle$,   $\langle x^2, -s \rangle$,  or to $\langle x \rangle$ (the acting torus is irreducible),   the six (-1)-curves form a unique $\Gamma$-orbit, hence we obtain a minimal surface that is a $\Bbbk$-form of $\text{Bl}_3(\Pkb^2)$ of Picard rank one. 
\item If the image of $\hat{\tau}$ is conjugated to $\langle x^2, s \rangle$ or to  $\langle x^2  \rangle$ (the acting torus is irreducible), the six (-1)-curves form two $\Gamma$-orbits. We can contract one of these two orbits and we obtain a minimal surface that is a $\Bbbk$-form of $\Pkb^2$ (of Picard rank one)  where the (1)-curves form a single $\Gamma$-orbits. 
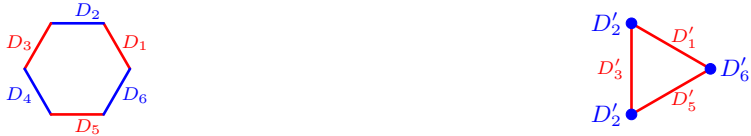
\begin{multicols}{2}
\begin{tikzpicture}[line cap=round,line join=round,>=triangle 45,x=0.7cm,y=0.7cm]
\clip(-2,-1.25) rectangle (2,1.25);
\draw [-,line width=1.pt, color=red] (1.,0.) -- (0.5,0.866);
\draw [-,line width=1.pt, color=blue] (0.5,0.866) -- (-0.5,0.866);
\draw [-,line width=1.pt, color=red] (-0.5,0.866) -- (-1.,0.);
\draw [-,line width=1.pt, color=blue] (-1.,0.) -- (-0.5,-0.866);
\draw [-,line width=1.pt, color=red] (-0.5,-0.866) -- (0.5,-0.866);
\draw [-,line width=1.pt, color=blue] (0.5,-0.866) -- (1.,0.);
\begin{scriptsize}
\draw[color=red] (1.1,0.5) node {$D_1$};
\draw[color=blue] (0.2,1.1) node {$D_2$};
\draw[color=red] (-1.1,0.5) node {$D_3$};
\draw[color=blue] (-1.1,-0.5) node {$D_4$};
\draw[color=red] (0.2,-1.1) node {$D_5$};
\draw[color=blue] (1.1,-0.5) node {$D_6$};
\end{scriptsize}
\end{tikzpicture}

\begin{tikzpicture}[line cap=round,line join=round,>=triangle 45,x=0.7cm,y=0.7cm]
\clip(-2,-1.25) rectangle (2,1.25);
\draw [-,line width=1.pt, color=red] (1.,0.) -- (-0.5,0.866);
\draw [-,line width=1.pt, color=red] (-0.5,0.866) -- (-0.5,-0.866);
\draw [-,line width=1.pt, color=red] (-0.5,-0.866) -- (1.,0.);
\filldraw[blue] (1,0) circle (2pt) node[anchor=west]{\small $D'_6$};
\filldraw[blue] (-0.5,0.866) circle (2pt) node[anchor=east]{\small $D'_2$};
\filldraw[blue] (-0.5,-0.866) circle (2pt) node[anchor=east]{\small $D'_2$};
\begin{scriptsize}
\draw[color=red] (0.5,0.6) node {$D'_1$};
\draw[color=red] (-0.9,0) node {$D'_3$};
\draw[color=red] (0.5,-0.6) node {$D'_5$};
\end{scriptsize}
\end{tikzpicture}
\end{multicols}
$\bullet$ If the image of $\hat{\tau}$  is conjugated to $\langle s, -s \rangle$ (the acting torus is irreducible), the six (-1)-curves form  two   $\Gamma$-orbits. We can contract the orbit  consisting of two (-1)-curves and we obtain a minimal surface that is a $\Bbbk$-form of $\Pkb^1 \times \Pkb^1$ where the (0)-curves form a single $\Gamma$-orbit, hence it is of  Picard rank one (the Picard group is generated by $D'_2 + D'_6 + D'_5 + D'_3$).
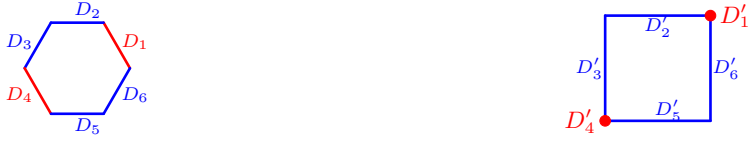
\begin{multicols}{2}
\begin{tikzpicture}[line cap=round,line join=round,>=triangle 45,x=0.7cm,y=0.7cm]
\clip(-2,-1.25) rectangle (2,1.25);
\draw [-,line width=1.pt, color=red] (1.,0.) -- (0.5,0.866);
\draw [-,line width=1.pt, color=blue] (0.5,0.866) -- (-0.5,0.866);
\draw [-,line width=1.pt, color=blue] (-0.5,0.866) -- (-1.,0.);
\draw [-,line width=1.pt, color=red] (-1.,0.) -- (-0.5,-0.866);
\draw [-,line width=1.pt, color=blue] (-0.5,-0.866) -- (0.5,-0.866);
\draw [-,line width=1.pt, color=blue] (0.5,-0.866) -- (1.,0.);
\begin{scriptsize}
\draw[color=red] (1.1,0.5) node {$D_1$};
\draw[color=blue] (0.2,1.1) node {$D_2$};
\draw[color=blue] (-1.1,0.5) node {$D_3$};
\draw[color=red] (-1.1,-0.5) node {$D_4$};
\draw[color=blue] (0.2,-1.1) node {$D_5$};
\draw[color=blue] (1.1,-0.5) node {$D_6$};
\end{scriptsize}
\end{tikzpicture}

\begin{tikzpicture}[line cap=round,line join=round,>=triangle 45,x=0.7cm,y=0.7cm]
\clip(-2.,-1.25) rectangle (2.,1.25);
\draw [-,line width=1.pt, color=blue] (1.,-1.) -- (1.,1.);
\draw [-,line width=1.pt, color=blue] (1.,1.) -- (-1.,1.);
\draw [-,line width=1.pt, color=blue] (-1.,1.) -- (-1.,-1.);
\draw [-,line width=1.pt, color=blue] (-1.,-1.) -- (1.,-1.);
\filldraw[red] (1,1) circle (2pt) node[anchor=west]{\small $D'_1$};
\filldraw[red] (-1,-1) circle (2pt) node[anchor=east]{\small $D'_4$};
\begin{scriptsize}
\draw[color=blue] (1.3,0) node {$D'_6$};
\draw[color=blue] (0.,0.8) node {$D'_2$};
\draw[color=blue] (-1.3,0.) node {$D'_3$};
\draw[color=blue] (0.2,-0.8) node {$D'_5$};
\end{scriptsize}
\end{tikzpicture}
\end{multicols}
$\bullet$ If the image of $\hat{\tau}$  is conjugated to $\langle x^3 \rangle$ (the acting torus is reducible), the six (-1)-curves form  three   $\Gamma$-orbits.   We can contract the orbit  consisting of two (-1)-curves and we obtain a minimal surface that is a $\Bbbk$-form of $\Pkb^1 \times \Pkb^1$ of Picard rank two where the (0)-curves have two $\Gamma$-orbits (the Picard group is generated by $D'_1 + D'_5$ and by $ D'_6 + D'_3$).  
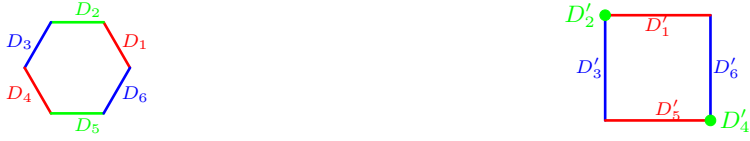
\begin{multicols}{2}
\begin{tikzpicture}[line cap=round,line join=round,>=triangle 45,x=0.7cm,y=0.7cm]
\clip(-2,-1.25) rectangle (2,1.25);
\draw [-,line width=1.pt, color=red] (1.,0.) -- (0.5,0.866);
\draw [-,line width=1.pt, color=green] (0.5,0.866) -- (-0.5,0.866);
\draw [-,line width=1.pt, color=blue] (-0.5,0.866) -- (-1.,0.);
\draw [-,line width=1.pt, color=red] (-1.,0.) -- (-0.5,-0.866);
\draw [-,line width=1.pt, color=green] (-0.5,-0.866) -- (0.5,-0.866);
\draw [-,line width=1.pt, color=blue] (0.5,-0.866) -- (1.,0.);
\begin{scriptsize}
\draw[color=red] (1.1,0.5) node {$D_1$};
\draw[color=green] (0.2,1.1) node {$D_2$};
\draw[color=blue] (-1.1,0.5) node {$D_3$};
\draw[color=red] (-1.1,-0.5) node {$D_4$};
\draw[color=green] (0.2,-1.1) node {$D_5$};
\draw[color=blue] (1.1,-0.5) node {$D_6$};
\end{scriptsize}
\end{tikzpicture}

\begin{tikzpicture}[line cap=round,line join=round,>=triangle 45,x=0.7cm,y=0.7cm]
\clip(-2.,-1.25) rectangle (2.,1.25);
\draw [-,line width=1.pt, color=blue] (1.,-1.) -- (1.,1.);
\draw [-,line width=1.pt, color=red] (1.,1.) -- (-1.,1.);
\draw [-,line width=1.pt, color=blue] (-1.,1.) -- (-1.,-1.);
\draw [-,line width=1.pt, color=red] (-1.,-1.) -- (1.,-1.);
\filldraw[green] (-1,1) circle (2pt) node[anchor=east]{\small $D'_2$};
\filldraw[green] (1,-1) circle (2pt) node[anchor=west]{\small $D'_4$};
\begin{scriptsize}
\draw[color=blue] (1.3,0) node {$D'_6$};
\draw[color=red] (0.,0.8) node {$D'_1$};
\draw[color=blue] (-1.3,0.) node {$D'_3$};
\draw[color=red] (0.2,-0.8) node {$D'_5$};
\end{scriptsize}
\end{tikzpicture}
\end{multicols}
$\bullet$ If the image of $\hat{\tau}$  is conjugated to $\langle s \rangle$ (the acting torus is irreducible), the six (-1)-curves form  four   $\Gamma$-orbits. We can contract an orbit   consisting of two (-1)-curves, then we contract the  invariant (-1)-curve and we obtain the minimal surface that is a $\Bbbk$-form of $\Pkb^2$ (of Picard rank one)  where the (1)-curves have two $\Gamma$-orbits.
The other possibility is to contract a $\Gamma$-invariant (-1)-curve. Then we can contract the pair of (-1)-curves or the invariant (-1)-curve. We obtain respectively a minimal surface that is a $\Bbbk$-form of $\Pkb^1 \times \Pkb^1$ of Picard rank one  (the Picard group is generated by $D'_3 + D'_5$) where the (0)-curves have two $\Gamma$-orbits, or a minimal surface that is a $\Bbbk$-form of $\Pkb^2$  where the (1)-curves have two $\Gamma$-orbits.
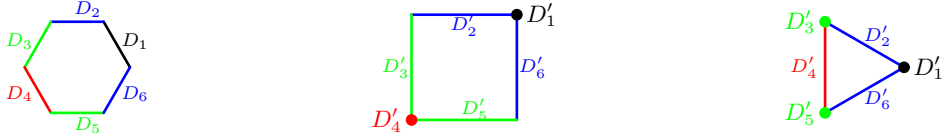
\begin{multicols}{3}
\begin{tikzpicture}[line cap=round,line join=round,>=triangle 45,x=0.7cm,y=0.7cm]
\clip(-2,-1.25) rectangle (2,1.25);
\draw [-,line width=1.pt, color=black] (1.,0.) -- (0.5,0.866);
\draw [-,line width=1.pt, color=blue] (0.5,0.866) -- (-0.5,0.866);
\draw [-,line width=1.pt, color=green] (-0.5,0.866) -- (-1.,0.);
\draw [-,line width=1.pt, color=red] (-1.,0.) -- (-0.5,-0.866);
\draw [-,line width=1.pt, color=green] (-0.5,-0.866) -- (0.5,-0.866);
\draw [-,line width=1.pt, color=blue] (0.5,-0.866) -- (1.,0.);
\begin{scriptsize}
\draw[color=black] (1.1,0.5) node {$D_1$};
\draw[color=blue] (0.2,1.1) node {$D_2$};
\draw[color=green] (-1.1,0.5) node {$D_3$};
\draw[color=red] (-1.1,-0.5) node {$D_4$};
\draw[color=green] (0.2,-1.1) node {$D_5$};
\draw[color=blue] (1.1,-0.5) node {$D_6$};
\end{scriptsize}
\end{tikzpicture}

\begin{tikzpicture}[line cap=round,line join=round,>=triangle 45,x=0.7cm,y=0.7cm]
\clip(-2.,-1.25) rectangle (2.,1.25);
\draw [-,line width=1.pt, color=blue] (1.,-1.) -- (1.,1.);
\draw [-,line width=1.pt, color=blue] (1.,1.) -- (-1.,1.);
\draw [-,line width=1.pt, color=green] (-1.,1.) -- (-1.,-1.);
\draw [-,line width=1.pt, color=green] (-1.,-1.) -- (1.,-1.);
\filldraw[black] (1,1) circle (2pt) node[anchor=west]{\small $D'_1$};
\filldraw[red] (-1,-1) circle (2pt) node[anchor=east]{\small $D'_4$};
\begin{scriptsize}
\draw[color=blue] (1.3,0) node {$D'_6$};
\draw[color=blue] (0.,0.8) node {$D'_2$};
\draw[color=green] (-1.3,0.) node {$D'_3$};
\draw[color=green] (0.2,-0.8) node {$D'_5$};
\end{scriptsize}
\end{tikzpicture}

\begin{tikzpicture}[line cap=round,line join=round,>=triangle 45,x=0.7cm,y=0.7cm]
\clip(-2,-1.25) rectangle (2,1.25);
\draw [-,line width=1.pt, color=blue] (1.,0.) -- (-0.5,0.866);
\draw [-,line width=1.pt, color=red] (-0.5,0.866) -- (-0.5,-0.866);
\draw [-,line width=1.pt, color=blue] (-0.5,-0.866) -- (1.,0.);
\filldraw[black] (1,0) circle (2pt) node[anchor=west]{\small $D'_1$};
\filldraw[green] (-0.5,0.866) circle (2pt) node[anchor=east]{\small $D'_3$};
\filldraw[green] (-0.5,-0.866) circle (2pt) node[anchor=east]{\small $D'_5$};
\begin{scriptsize}
\draw[color=blue] (0.5,0.6) node {$D'_2$};
\draw[color=red] (-0.9,0) node {$D'_4$};
\draw[color=blue] (0.5,-0.6) node {$D'_6$};
\end{scriptsize}
\end{tikzpicture}
\end{multicols}
$\bullet$ If the image of $\hat{\tau}$  is conjugated to $\langle id \rangle$ (the acting torus is reducible),   the six (-1)-curves form six   $\Gamma$-orbit, hence we can contract and we obtain a minimal surface that is a	 $\Bbbk$-form of $\Pkb^2$ (of Picard rank one) where the (1)-curves have three $\Gamma$-orbits. 
\end{proof}

\begin{example}[{Another proof of $H^1( \mathcal{G}_{12}, \Aut^{\T}(\T)) = \{ 1 \}$}] \label{ExampleH1G12} 
Let $(\T, \tau)$ be a $\Bbbk$-torus such that the image of the $\Gamma$-representation $\hat{\tau}$ on $N$ is $\mathcal{G}_{12} = \langle s \rangle$. For instance $\Bbbkb=\C$ and $\Bbbk=\R$ and $(\T, \tau)$ is the Weil Restriction $\R$-torus. Note that, by \cite[Proposition 3.7]{Elizondo}, we can assume that $\T=\Spec(\Bbbk'[M])$, where $\Bbbk'/\Bbbk$ is a finite Galois extension of Galois group isomorphic to $\mathcal{G}_{12}$. Let $(\T, \sigma)$ be a $(\T, \tau)$-torsor. We will prove that this torsor is trivial.  The fan of the toric Del Pezzo surface $\Pkp^2$ is Galois stable, therefore, by \cite[Proposition 1.19]{Huruguen}, $(\T, \sigma)$ is a dense $(\T, \tau)$-stable open subset of a Severi Brauer $\Bbbk$-surface $(\Pkp^2, \sigma')$. The Galois action on the fan exchange two rays and fix the other one; therefore the Galois action on the toric divisors exchanges two toric divisors and fix a point $P$. Thus, $(\Pkp^2, \sigma')$ contains a $\Bbbk$-point $P$. A Severi Brauer suface with a $\Bbbk$-point is isomorphic to $\Pk^2$. Therefore, $(\T, \sigma)$ is a trivial $(\T, \tau)$-torsor since it contains $\Bbbk$-points, and $H^1( \mathcal{G}_{12}, \Aut^{\T}(\T)) = \{ 1 \}$.
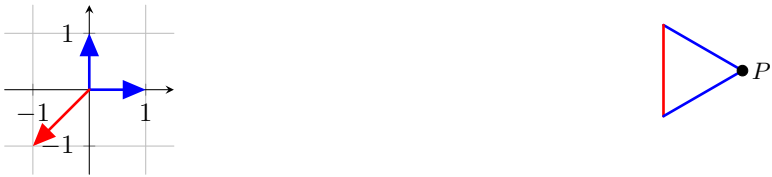
\begin{multicols}{2}
\begin{tikzpicture}[line cap=round,line join=round,>=triangle 45,x=0.7cm,y=0.7cm]
\begin{axis}[
x=0.75cm,y=0.75cm,
axis lines=middle,
ymajorgrids=true,
xmajorgrids=true,
xmin=-1.5,
xmax=1.5,
ymin=-1.5,
ymax=1.5,
xtick={-3.0,-2.0,...,3.0},
ytick={-3.0,-2.0,...,3.0},]
\clip(-3.,-3.) rectangle (3.,3.);
\draw [->,line width=1.pt, color=blue] (0.,0.) -- (1.,0.);
\draw [->,line width=1.pt, color=blue] (0.,0.) -- (0.,1.);
\draw [->,line width=1.pt, color=red] (0.,0.) -- (-1.,-1.);
\end{axis}
\end{tikzpicture}

\begin{tikzpicture}[line cap=round,line join=round,>=triangle 45,x=0.7cm,y=0.7cm]
\clip(-2,-1.25) rectangle (2,1.25);
\draw [-,line width=1.pt, color=blue] (1.,0.) -- (-0.5,0.866);
\draw [-,line width=1.pt, color=red] (-0.5,0.866) -- (-0.5,-0.866);
\draw [-,line width=1.pt, color=blue] (-0.5,-0.866) -- (1.,0.);
\filldraw[black] (1,0) circle (2pt) node[anchor=west]{\small $P$};
\end{tikzpicture}
\end{multicols}
\end{example}

\begin{example}
Let $\Bbbk=\R$, and let $\Bbbk'= \C$. 
\begin{enumerate}[leftmargin=0.75cm, label=(\roman*)]
\item Consider the trivial $\s^1 \times \s^1$-torsor. Then, it is an $\s^1 \times \s^1$-stable dense open subset of $\PR^1 \times \PR^1 \cong \Proj(\R[x,y,z,t]/(x^2+y^2-z^2-t^2))$  (by Segre embedding).
\item Consider the non trivial $\s^1 \times \s^1$-torsor $\Spec(\R[x,y]/(x^2+y^2+1))^{2}$. Then, it is an $\s^1 \times \s^1$-stable dense open subset of $\Proj(\R[x,y,z,t]/(x^2+y^2+z^2+t^2))$, which is an $\R$-form of  $\PR^1 \times \PR^1$ with no $\R$-points. 
\item Consider the trivial $\RC(\GC)$-torsor. Then, it is an $\RC(\GC)$-stable dense open subset of $\RC(\PC^1) \cong \Proj(\R[x,y,z,t]/(x^2-y^2-z^2-t^2))$, which is an $\R$-form of  $\PR^1 \times \PR^1$ with $\R$-points.
\end{enumerate} 
\end{example}

\begin{example} 
Let $\Bbbk$ be a field, and let $\Bbbk'/\Bbbk$ be a cubic Galois extension of Galois group $\{id, \gamma, \gamma^2\}$.
Then, a $\text{R}^{(1)}_{\Bbbk'/\Bbbk}(\Gkp)$-torsor $X_{\alpha}$ corresponds to the $\Bbbk$-structure $\sigma$ on $\Gkp^2$ defined by:
$$\sigma_{\gamma}^{\sharp} \left(a_{(k,l)}\chi^{(k,l)} \right) = \gamma(a_{(k,l)}) h_{\gamma} \left(\tilde{\tau}_{\gamma}(k,l) \right) \chi^{\tilde{\tau}_{\gamma}(k,l)}= \gamma(a_{(k,l)}) \alpha^{-l} \chi^{(-l,k-l)},$$
where $\alpha \in \Bbbk^*$. If $\alpha \nsubseteq \mathrm{Im}(N_{\Bbbk'/\Bbbk}(\Bbbk))$, then $X_{\alpha}$ is  a non-trivial  $\text{R}^{(1)}_{\Bbbk'/\Bbbk}(\Gkp)$-torsor (see Lemma \ref{H1Norm1}), which is a  $\text{R}^{(1)}_{\Bbbk'/\Bbbk}(\Gkp)$-stable dense open subset of a $\Bbbk$-form of $\Pkp^2$ with no $\Bbbk$-points.
\end{example}

\begin{example}
Let $\Bbbk'/\Bbbk$ be a Galois extension with Galois group $\Gamma  \cong \mathscr{D}_{12} \cong \langle x , s \rangle$. Let $(\Gkp^2, \tau)$ be the $\Bbbk$-torus such that $\hat{\tau}_{x}= x$ and $\hat{\tau}_{s}=s$. 
Consider the surface $X$ of $\Pkp^2 \times \Pkp^2$:
$$X:= \left\{ \left([x_1:x_2:x_3],[y_1:y_2:y_3]\right) \ | \ x_i, y_j \in \Bbbk', \ x_iy_i = x_jy_j, \ i \in \{1,2,3\} \right\}.$$ 
It is the graph of the standard quadratic birational Cremona transformation. This defines a Del Pezzo $\Bbbk'$-surface of degree six. Recall that we have an isomorphism $\mathscr{D}_{12} \cong \mathscr{S}_3 \times \mathscr{C}_2$. 
Let $\sigma : \mathscr{S}_3 \times \mathscr{C}_2 \to \Aut(X_{\Bbbk})$ be the $\Bbbk$-structure on $X$ defined as follow. The group $\mathscr{S}_3$ acts on $\Pkp^2 \times \Pkp^2$ by permuting separately the elements of the pair of triples, and by the Galois action on the constants. The group $\mathscr{C}_2$ acts on $\Pkp^2 \times \Pkp^2$ by exchanging the pair of triples, and by the Galois action on the constants. This defines a $\Bbbk$-structure on $X$ since the equations of $X$ are invariant. Then, one can shows that the $\Bbbk$-torus $(\Gkp^2, \tau)$ acts on $(X, \sigma)$, therefore  $(X, \sigma)$ is a toric Del pezzo $\Bbbk$-surface. Note that the open orbit is a trivial $(\Gkp^2, \tau)$-torsor since it contains an invariant point $([1:1:1],[1:1:1])$. 
\end{example}

\bigskip

\appendix
\addappheadtotoc 

\addtocontents{toc}{\protect\setcounter{tocdepth}{0}}

\section{$\Bbbk$-structures} \label{Appendix1}

In this section, we explain the link between infinite and finite Galois descent. Let $\Bbbk'/\Bbbk$ be a Galois extension of Galois group $\Gamma$. In Lemmas \ref{LinkFinite} and \ref{ForAll}, we explain how are related $\Bbbk$-structures on a $\Bbbk'$-variety $X$ and Galois subextensions $\Bbbk \subset \Bbbk_1 \subset \Bbbk'$. Let $\gamma \in \Gamma$. Since $\Bbbk_1$ is a normal subextension in $\Bbbk'$, $\gamma(\Bbbk_1)=\Bbbk_1$. We get an automorphism $\left. \gamma \right|_{\Bbbk_1} : \Bbbk_1 \to \Bbbk_1$ (\cite[Lemmas 09HQ and 0BME]{Stack}). Therefore, we obtain a continuous surjective homomorphism $\Gamma \to \Gal(\Bbbk_1/\Bbbk), \gamma \mapsto \left. \gamma \right|_{\Bbbk_1}$, with kernel $\Gal(\Bbbk'/\Bbbk_1)$ (\cite[Lemmas 0BMK, 0BMM and Theorem 0BML]{Stack}). In other words, we obtain a short exact sequence 
$$1 \longrightarrow  \Gal \left( \Bbbk'/\Bbbk_1 \right) \longrightarrow \Gamma := \Gal \left( \Bbbk'/\Bbbk \right) \longrightarrow \Gal \left( \Bbbk_1/\Bbbk \right) \longrightarrow 1$$
of profinite topological groups. Hence, $\Gal(\Bbbk'/\Bbbk_1)$ is a  normal subgroup of $\Gamma$ and we have an isomorphism $\Gamma /\Gal(\Bbbk'/\Bbbk_1) \to \Gal(\Bbbk_1/\Bbbk), \gamma \Gal(\Bbbk'/\Bbbk_1) \mapsto \left. \gamma \right|_{\Bbbk_1}$.
 
\begin{lemma} \label{LinkFinite}
Let $X$ be a  $\Bbbk'$-variety endowed with a $\Bbbk$-structure $\sigma$. Let $\Bbbk_1/\Bbbk$ be a  finite Galois extension and let $X_1$ be a $\Bbbk_1$-form of $X$  as in Definition \ref{defstructure}. There exists a $\Bbbk$-structure $\sigma_1 : \Gal(\Bbbk_1/\Bbbk)   \to \Aut({X_1}_{/\Bbbk})$   such that the following diagram commutes for all $\gamma \in \Gamma$:
\begin{center}
\begin{tikzpicture}
  \matrix (m) [matrix of math nodes,row sep=1.5em,column sep=2.25em,minimum width=2em]
  {
 X_1 \times_{\Spec \left( \Bbbk_1 \right)} \Spec \left( \Bbbk' \right) & X \\
 X_1 \times_{\Spec \left( \Bbbk_1 \right)} \Spec \left( \Bbbk' \right) & X  \\};
  \path[-stealth]
    (m-1-1) edge node [above] {\small $\cong    $} (m-1-2)
		        edge node [left] {\small ${\sigma_1}_{\left. \gamma \right|_{\Bbbk_1}} \times \Spec(\gamma)$} (m-2-1)
		(m-2-1) edge node [above] {\small $\cong    $} (m-2-2)
    (m-1-2) edge node [right] {\small ${\sigma}_{\gamma}$} (m-2-2);
\end{tikzpicture}
\end{center}
Furthermore, if $X$ is quasi-projective (resp. affine), there exists a $\Bbbk$-form $X_0$ of $X$ in the category of quasi-projective (resp. affine) varieties, and a $\Gamma$-equivariant isomorphism $(X_0 )_{\Bbbk'} \cong X$, where the $\Bbbk$-structure on $X$ is $\sigma$ and the $\Bbbk$-structure on $(X_0)_{\Bbbk'}$ is given by $\gamma \in \Gamma \mapsto id \times \Spec(\gamma)$. 
\end{lemma}

\begin{proof} 
We give a sketch of the proof. The construction of  the $\Bbbk$-structure $\sigma_1$ is based on fpqc descent (see \cite[Theorem 6]{Neron}). 
If $X$ is quasi-projective (resp. affine), then so is $X_1$ by  \cite[Proposition 14.53 and 14.57]{Gortz}. Then, we use classical results of finite Galois  descent: by \cite[Corollaire 7.7]{SGA1}, the desired $\Bbbk$-form $X_0$ of $X$  exists.  
\end{proof}

\begin{lemma} \label{ForAll}
If $\sigma$ is a $\Bbbk$-structure on a quasi-projective $\Bbbk'$-variety $X$, then for all  finite  Galois extension $\Bbbk_1/\Bbbk$ in $\Bbbk'$ there exists a $\Bbbk_1$-form $X_1$ of the $\Bbbk'$-variety $X$ such that the restriction of $\sigma$ to $\Gal(\Bbbk'/\Bbbk_1)$ coincides with the natural $\Gal(\Bbbk'/\Bbbk_1)$-action on $(X_1)_{\Bbbk'}  \cong X$. 
\end{lemma}

\begin{proof}
We give a sketch of the proof.  
Since $\sigma$ is a $\Bbbk$-structure on $X$, there exists a finite Galois extension $\Bbbk_1 / \Bbbk$ in $\Bbbk'$   and a $\Bbbk_1$-form $X_1$ of the $\Bbbk'$-variety $X$. 
Let $\Bbbk_2/\Bbbk$ be another finite Galois extension in $\Bbbk'$.  There exists a finite Galois extension $\Bbbk_3/\Bbbk$  that contains $\Bbbk_1$ and $\Bbbk_2$   \cite[Lemmas 0EXM and 09DT]{Stack}. We obtain a $\Bbbk_3$-form $X_3 := (X_1)_{\Bbbk_3}$ of $X$, and by Lemma \ref{LinkFinite}, there exists a $\Bbbk$-structure $\sigma_3$ on $X_3$. Since $\Bbbk_3/\Bbbk_2$ is a finite Galois extension, by \cite[Corollaire 7.7]{SGA1}, the desired $\Bbbk_2$-form $X_2$ of $X$  exists.  
\end{proof}

\section{Tori} \label{Appendix2}

\subsection{Factorization by a finite Galois extension that splits the $\Bbbk$-torus} \label{Appendix2Split}

We pursue Remark \ref{representation}. Let $\Bbbk_1/\Bbbk$ be a finite Galois extension that splits $(\T, \tau)$ and let $H:= \Gal(\Bbbk'/\Bbbk_1)$ be the absolute Galois group of $\Bbbk'/\Bbbk_1 $.  By Lemmas \ref{LinkFinite} and \ref{ForAll}, there exists $\Bbbk_1$-form $\T_1$ of $\T$ and a $\Bbbk$-group structure $\tau_1$ on $\T_1$ such that the following diagram commutes for all $\gamma \in \Gamma$
\begin{center}
\begin{tikzpicture}
  \matrix (m) [matrix of math nodes,row sep=1.5em,column sep=3em,minimum width=2em]
  {
T_1 \times_{\Spec \left( \Bbbk_1 \right)} \Spec \left( \Bbbk' \right) & \T \\
T_1 \times_{\Spec \left( \Bbbk_1 \right)} \Spec \left( \Bbbk' \right) & \T \\
};
  \path[-stealth]
    (m-1-1) edge node [above] {\small $ \cong    $} (m-1-2)
            edge  node [left]  {\small ${\tau_1}_{\gamma |_{\Bbbk_1}} \times \Spec(\gamma)$ } (m-2-1)
    (m-2-1) edge node [above] {\small   $\cong $ } (m-2-2)
		(m-1-2) edge node [right] { \small  $\tau_{\gamma} $ } (m-2-2);
\end{tikzpicture}
\end{center}
Since $\Bbbk_1$ splits $(\T, \tau)$, we can assume $\T_1 = \Spec(\Bbbk_1[M])$. Therefore, the $H$-action on $M$ is trivial and the $\Gamma$-action on $M$  factorizes:
\begin{center}
\begin{tikzpicture}
  \matrix (m) [matrix of math nodes,row sep=0.75em,column sep=3em,minimum width=2em]
  {
              \Gamma          &             \Gl(M)               \\
               \Gamma /  H  \cong \Gal \left( \Bbbk_1/\Bbbk \right)       &                         \\};
  \path[-stealth]
    (m-1-1) edge node [above] {\small $\tilde{\tau}$} (m-1-2)
		        edge   (m-2-1)   
    (m-2-1) edge node [below] {\small $\tilde{\tau_1}$} (m-1-2);
\end{tikzpicture}
\end{center}
This means that   the profinite group $\Gamma$ acts on $M$ as the finite Galois group $\Gal(\Bbbk_1/\Bbbk)$, and this action comes from the $\Gal( \Bbbk_1/\Bbbk)$-action on $\T_1$. Furthermore, since the kernel of $\hat{\tau}_1$ is a normal subgroup of $\Gal(\Bbbk_1/\Bbbk)$, by the Galois correspondence, there exists a finite Galois extension $\Bbbk_2/\Bbbk$ in $\Bbbk_1$ such that the induced map $\Gal(\Bbbk_2/\Bbbk) \cong \Gal(\Bbbk_1/\Bbbk)/\text{Ker}(\hat{\tau}_1) \hookrightarrow \Gl(M)$ is injective. Therefore, the corresponding $\Bbbk_2$-torus $\T_2$ is split and endowed with a $\Bbbk$-group structure $\tau_2$ as in Lemma \ref{ForAll}.

\subsection{Morphisms of tori and exact sequences of lattices} \label{Appendix2Morphism}

Let  $(\T = \Spec(\Bbbk'[M]), \tau)$ be a subtorus of the $\Bbbk$-torus $( \Gkp^n = \Spec(\Bbbk'[M']), \tau')$. The inclusion $\T \hookrightarrow \Gkp^n$ induces   a surjective lattice homomorphism $M' \to M$.
 Let  ${M}_Y$ be the kernel of this homomorphism, it is  a sublattice of $M'$.  Moreover, the $\Gamma$-action  $\tilde{\tau}'$ on $M'$ induces a $\Gamma$-action $\tilde{\tau}_Y$ on $M_Y$. Let ${\tau}_Y$ be the induced $\Bbbk$-group structure on $\T_Y = \Spec(\Bbbk'[M_Y])$. The next diagram of algebraic $\Bbbk'$-groups commutes for all $\gamma \in \Gamma$:
\begin{center}
\begin{tikzpicture}
\matrix (m) [matrix of math nodes,row sep=0.7em,column sep=5em,minimum width=2em]  {
	1^{\vphantom{n}}_{\vphantom{,}}   & \T^{\vphantom{n}}_{\vphantom{,}}   & \Gkp^n & {\T_Y}^{\vphantom{n}}_{\vphantom{,}}  & 1^{\vphantom{n}}_{\vphantom{,}}    \\
 	1^{\vphantom{n}}_{\vphantom{,}}   & \T^{\vphantom{n}}_{\vphantom{,}}   & \Gkp^n & {\T_Y}^{\vphantom{n}}_{\vphantom{,}}  & 1^{\vphantom{n}}_{\vphantom{,}}    \\};
 \path[-stealth]
  	(m-1-1) edge    (m-1-2)
		(m-1-2) edge    (m-1-3)
		(m-1-3) edge    (m-1-4)
		(m-1-4) edge    (m-1-5)
		(m-2-1) edge    (m-2-2)
		(m-2-2) edge    (m-2-3)
		(m-2-3) edge    (m-2-4)
		(m-2-4) edge    (m-2-5)
		(m-1-2) edge node [right] {\small  $ {\tau}_{\gamma} $ }   (m-2-2)
		(m-1-3) edge node [right] {\small  $ {\tau'}_{\gamma} $ }   (m-2-3)
		(m-1-4) edge node [right] {\small  $ {\tau_Y}_{\gamma} $ }   (m-2-4) ;
	\end{tikzpicture}
\end{center} 
There exists an injective morphism $F : N \to N'$ and  a surjective homomorphism $P : N' \to N_Y$, and the following   diagrams of free $\Z$-modules commute for all $\gamma \in \Gamma$:  
\begin{center}
\begin{tikzpicture}
  \matrix (m) [matrix of math nodes,row sep=0.7em,column sep=7em,minimum width=2em]
  {
	0_{\vphantom{Y}}{\vphantom{'}}   & N_{\vphantom{Y}}{\vphantom{'}}   & N'_{\vphantom{Y}}  & N_Y{\vphantom{'}} & 0_{\vphantom{Y}}{\vphantom{'}}     \\
	0_{\vphantom{Y}}{\vphantom{'}}   & N_{\vphantom{Y}}{\vphantom{'}}    & N'_{\vphantom{Y}}  & N_Y{\vphantom{'}} & 0_{\vphantom{Y}}{\vphantom{'}}     \\};
 \path[-stealth]
   	(m-1-1) edge    (m-1-2)
		(m-1-2) edge node [above] {\small $F$ }   (m-1-3)
		(m-1-3) edge node [above] {\small $P$ }     (m-1-4)
		(m-1-4) edge    (m-1-5)
		(m-2-1) edge    (m-2-2)
		(m-2-2) edge  node [above] {\small $F$ }    (m-2-3)
		(m-2-3) edge  node [above] {\small $P$ }    (m-2-4)
		(m-2-4) edge    (m-2-5)
		(m-1-2) edge node [right] {\small  $  \hat{\tau }_{\gamma}$ }   (m-2-2)
		(m-1-3) edge node [right] {\small  $ \hat{\tau}'_{\gamma}  $ }   (m-2-3)
		(m-1-4) edge node [right] {\small  $ \hat{\tau_Y}_{\gamma}$ }   (m-2-4)
 ;
	\end{tikzpicture}
\end{center}
\begin{center}
\begin{tikzpicture}
  \matrix (m) [matrix of math nodes,row sep=0.7em,column sep=7em,minimum width=2em]
  {
	0_{\vphantom{Y}}{\vphantom{'}}   & M_Y{\vphantom{'}}  & M'_{\vphantom{Y}}  & M_{\vphantom{Y}}   & 0_{\vphantom{Y}}{\vphantom{'}}     \\
	0_{\vphantom{Y}}{\vphantom{'}}   & M_Y{\vphantom{'}}  & M'_{\vphantom{Y}}  & M_{\vphantom{Y}}   & 0_{\vphantom{Y}}{\vphantom{'}}     \\};
 \path[-stealth]
   	(m-1-1) edge    (m-1-2)
		(m-1-2) edge   node [above] {\small $P^*$ }   (m-1-3)
		(m-1-3) edge   node [above] {\small $F^*$ }  (m-1-4)
		(m-1-4) edge    (m-1-5)
		(m-2-1) edge    (m-2-2)
		(m-2-2) edge   node [above] {\small $P^*$ }    (m-2-3)
		(m-2-3) edge   node [above] {\small $F^*$ }    (m-2-4)
		(m-2-4) edge    (m-2-5)
		(m-1-2) edge node [right] {\small  $ \tilde{\tau_Y}_{\gamma}$ }   (m-2-2)
		(m-1-3) edge node [right] {\small  $ \tilde{\tau}'_{\gamma}  $ }   (m-2-3)
		(m-1-4) edge node [right] {\small  $ \tilde{\tau}_{\gamma}$ }   (m-2-4)
 ;
	\end{tikzpicture}
\end{center}
There always exists a section $s^* : M \to M' $ (i.e. $F^* \circ s^* =id_M$)  \cite[Proposition A.3.1]{Neron}, but not always a $\Gamma$-equivariant one. Therefore, we obtain a section $\T_Y \to \Gkp^n $, but not always a $\Gamma$-equivariant one. In other words, $\Gkp^n \cong \T \times \T_Y$, but this isomorphism is not always $\Gamma$-equivariant.

\section{Torus actions} \label{Appendix2TorusActions}

The next propositions is a direct consequence of fpqc descent in the category of varieties endowed with a torus action (see \cite[Corollaire 7.9]{SGA1}). 

\begin{proposition}  \label{defaction}
Let $T $ be a $\Bbbk$-torus. There is a one-to-one correspondence between     quasi-projective $\Bbbk$-varieties    endowed with a $T$-action  and   tuples $(\T, \tau , X, \sigma , \mu)$   consisting of: 
\begin{enumerate}[leftmargin=0.75cm, label=(\roman*)]
\item   a $\Bbbk'$-torus $\T$ endowed with a $\Bbbk$-group structure $\tau$ such that $ \T/\{ \tau_{\gamma} \ | \ \gamma \in \Gamma \} \cong T $;
\item a quasi-projective $\Bbbk'$-variety $X$ endowed with a $\Bbbk$-structure $\sigma$; 
\item    an  action $\mu : \T \times X \to X$   such that the following diagram commutes for all $\gamma \in \Gamma$:
\begin{center}
\begin{tikzpicture}
\matrix (m) [matrix of math nodes,row sep=1em,column sep=7em,minimum width=2em]
 { \T \times  X   &    X   \\
   \T \times  X   &    X   \\};
 \path[-stealth]
    (m-1-1) edge node  [above]{\small $\mu $} (m-1-2)
            edge node  [left]{\small ${\tau}_{\gamma}  \times {\sigma }_{\gamma}  $} (m-2-1)
    (m-1-2) edge node  [right]{\small ${\sigma_{\gamma} } $}  (m-2-2)
    (m-2-1) edge node  [below]{\small $\mu$} (m-2-2);
\end{tikzpicture}
\end{center}
\end{enumerate}
\end{proposition}

Let  $\T=\Spec(\Bbbk[M])$  be a split torus acting on an affine variety $X$.  Recall that the coordinate algebra $\Bbbk[X]$ of $X$ is $M$-graded. From Proposition \ref{defaction}, we get the next lemma that is used in the proof of Proposition \ref{ToricDown}.

\begin{lemma} \label{IsomGradedStructure}  
Let $(\T = \Spec(\Bbbk'[M]), \tau)$ be a $\Bbbk$-torus  acting   on the affine   $\Bbbk$-variety $(X , \sigma )$. Let  $\omega_M$ be the weight cone of the $\T$-action on $X$.  Then $\tilde{\tau}(\omega_M) = \omega_M$ and for all $m \in M$ and $\gamma \in \Gamma$
\begin{equation*}
{\sigma}_{\gamma}^{\sharp} \left(\Bbbk'[X]_m \right) = \Bbbk'[X]_{\tilde{\tau}_{\gamma}(m)}.
\end{equation*}
\end{lemma}

\begin{proof}    
Let  $m \in M$, $\gamma \in \Gamma$, and let $f \in \Bbbk'[X]_m$. We obtain from the diagram of Proposition \ref{defaction}; 
\begin{equation*}
 \left( \mu^{\sharp}  \circ {\sigma }_{\gamma}^{\sharp}  \right)(f)  =  \left(  \left( {\tau}_{\gamma}^{\sharp}  \times {\sigma}_{\gamma}^{\sharp}  \right) \circ \mu^{\sharp} \right)(f) 
=  {\tau}_{\gamma}^{\sharp}  \left( \chi^m \right) \otimes  {\sigma}_{\gamma}^{\sharp}  (f) 
=   \chi^{\tilde{\tau}_{\gamma}(m)}  \otimes   {\sigma}_{\gamma}^{\sharp} (f). 
\end{equation*}
Hence  ${\sigma}_{\gamma}^{\sharp}(\Bbbk'[X]_m) \subset \Bbbk'[X]_{\tilde{\tau}_{\gamma}(m)} $.  Moreover, if $g \in \Bbbk'[X]_{\tilde{\tau}_{\gamma}(m)}$, then $g = {\sigma}_{\gamma}^{\sharp}({\sigma}_{\gamma^{-1}}^{\sharp}(g))$.
\end{proof}


\end{document}